\documentclass[a4paper, 12pt, twoside, reqno, openright]{amsart}
\usepackage{amsmath}
\usepackage{amssymb}
\usepackage{amsthm}
\usepackage{amscd}
\usepackage{amsopn}
\usepackage{paralist}
\usepackage[UKenglish]{babel}
\usepackage{courier}
\usepackage[T1]{fontenc}
\usepackage{multicol}
\usepackage{pifont}
\usepackage[top=1.5in, bottom=1.2in, left=1.2in, right=1.2in]{geometry}
\usepackage{wrapfig}
\usepackage{etex}
\usepackage{pictex}
\usepackage[all,cmtip]{xy}
\usepackage[usenames,dvipsnames]{xcolor}
\usepackage[cal=boondox, scr=euler, bb=ams]{mathalfa}

\newtheoremstyle{Teorema}{5pt}{5pt}{\it}{}{\bf}{.}{ }{}
\theoremstyle{Teorema}
\newtheorem{Theorem}{Theorem}[section]
\newtheorem*{ThrmAOne}{Theorem A1}
\newtheorem*{ThrmATwo}{Theorem A2}
\newtheorem*{ThrmB}{Theorem B}
\newtheorem*{ThrmC}{Theorem C}

\newtheorem{Corollary}[Theorem]{Corollary}
\newtheorem{Proposition}[Theorem]{Proposition}
\newtheorem{Definition}[Theorem]{Definition}
\newtheorem{PropDef}[Theorem]{Proposition and Definition}
\newtheorem{Lemma}[Theorem]{Lemma}
\newtheoremstyle{Annotazione}{5pt}{5pt}{\rm}{}{\it}{.}{ }{}
\theoremstyle{Annotazione}
\newtheorem{Remark}[Theorem]{Remark}
\newtheorem{Condition}[Theorem]{Condition}
\newtheorem{Example}[Theorem]{Example}
\newtheorem{Examples}[Theorem]{Examples}

\def\SL{\mathrm{SL}}

\def\PGL{\mathrm{PGL}}
\def\GL{\mathrm{GL}}

\def\Sh{\mathrm{Sh}}
\def\bba{\mathbb{A}}
\def\bbz{\mathbb{Z}}
\def\bbr{\mathbb{R}}
\def\bbq{\mathbb{Q}}
\def\bbc{\mathbb{C}}
\def\bbh{\mathbb{H}}

\def\bbp{\mathbb{P}}

\def\bbn{\mathbb{N}}

\def\bbv{\mathbb{V}}
\def\bbw{\mathbb{W}}
\def\bbs{\mathbb{S}}
\def\bbg{\mathbb{G}}
\def\bbt{\mathbb{T}}

\def\id{\mathrm{id}}

\def\Res{\operatorname{Res}}
\def\Pic{\operatorname{Pic}}
\def\End{\operatorname{End}}
\def\Hom{\operatorname{Hom}}
\def\Aut{\operatorname{Aut}}
\def\tr{\operatorname{tr}}

\def\Gal{\operatorname{Gal}}

\def\Spec{\operatorname{Spec}}
\def\Res{\operatorname{Res}}

\def\an{\mathrm{an}}
\def\can{\mathrm{can}}

\def\qbar{\widebar{\bbq}}
\def\Cbar{\widebar{C}}

\def\defined{\overset{\mathrm{def}}{=}}
\def\Hup{\mathrm{H}}
\def\tit{\tilde{T}}
\def\tif{\tilde{f}}
\def\taux{{^{\tau ,x}}}
\def\tauy{{^{\tau ,y}}}
\def\ad{\mathrm{ad}}
\def\sp{\operatorname{sp}}
\def\gad{G^{\mathrm{ad}}}
\def\had{H^{\mathrm{ad}}}

\def\et{\text{\rm \'{e}t}}

\makeatletter
\newcommand*\rel@kern[1]{\kern#1\dimexpr\macc@kerna}
\newcommand*\widebar[1]{%
  \begingroup
  \def\mathaccent##1##2{%
    \rel@kern{0.8}%
    \overline{\rel@kern{-0.8}\macc@nucleus\rel@kern{0.2}}%
    \rel@kern{-0.2}%
  }%
  \macc@depth\@ne
  \let\math@bgroup\@empty \let\math@egroup\macc@set@skewchar
  \mathsurround\z@ \frozen@everymath{\mathgroup\macc@group\relax}%
  \macc@set@skewchar\relax
  \let\mathaccentV\macc@nested@a
  \macc@nested@a\relax111{#1}%
  \endgroup
}

\makeatother

\begin{document}

\title{Modular embeddings and automorphic~Higgs~bundles}
\author{Robert A. Kucharczyk}
\address{Universit\"{a}t Bonn\\ Mathematisches Institut\\ Endenicher Allee 60\\ 53115 Bonn\\ Germany}
\email{robert.a.kucharczyk@googlemail.com}
\thanks{Research partially supported by the European Research Council}
\begin{abstract}
We investigate modular embeddings for semi-arithmetic Fuchsian groups. First we prove some purely algebro-geometric or even topological criteria for a regular map from a smooth complex curve to a quaternionic Shimura variety to be covered by a modular embedding. Then we set up an adelic formalism for modular embeddings and apply our criteria to study the effect of abstract field automorphisms of $\bbc$ on modular embeddings. Finally we derive that the absolute Galois group of $\bbq$ operates on the dessins d'enfants defined by principal congruence subgroups of (arithmetic or non-arithmetic) triangle groups by permuting the defining ideals in the tautological way.
\end{abstract}
\maketitle

\tableofcontents

\thispagestyle{empty}

\section{Introduction}

\noindent 
In this article we examine several geometric and arithmetic aspects of modular embeddings in the sense of Schmutz Schaller and Wolfart~\cite{MR1745404}. These are certain holomorphic maps $\tif\colon\bbh\to\bbh^r$, where $\bbh$ is the complex upper half plane, which are equivariant for group homomorphisms $\varphi\colon\varDelta\to\varGamma$, where $\varDelta\subset\PGL_2(\bbr )^+$ is a (not necessarily arithmetic) lattice and $\varGamma\subset (\PGL_2(\bbr )^+)^r$ is an irreducible arithmetic lattice, and where the first projection $\PGL_2(\bbr )^r\to\PGL_2(\bbr )$ composed with $\varphi$ induces the identity on~$\varDelta$. Such embeddings induce regular maps between algebraic varieties $f\colon C\to S$, where $C$ is the algebraic curve with $C^{\an }=\varDelta\backslash\bbh$, and $S$ is the Shimura variety with $S^{\an }=\varGamma\backslash\bbh^r$.

As of today three series of examples are known for groups $\varDelta$ admitting modular embeddings:
\begin{enumerate}
\item If $\varDelta$ is \emph{arithmetic} the identity $\bbh\to\bbh$ is a modular embedding. More generally modular embeddings into higher-dimensional quaternionic Shimura varieties can be constructed quite easily.
\item All Fuchsian \emph{triangle groups} $\varDelta =\varDelta_{p,q,r}$ admit modular embeddings, but only finitely many of them are arithmetic (see \cite{MR0429744}). These modular embeddings were constructed in \cite{MR1075639} where also modular embeddings for non-arithmetic groups were first introduced.
\item Those \emph{Veech groups} that are lattices -- in which case they can also be characterised as the uniformising groups of \emph{Teichm\"{u}ller curves} -- admit modular embeddings. This was discovered in \cite{MR1992827} for genus two and in \cite{MR2188128} for the general case.
\end{enumerate}
See Examples~\ref{ExamplesME} below for a more detailed review.

The main results of this article can be roughly divided into two distinct but related groups. In the first group we start with an arbitrary regular map $f\colon C\to S$ where $C$ is an algebraic curve over $\bbc$ and $S$ is a Shimura variety of the kind that appears for modular embeddings (a quaternionic connected Shimura variety, see Section~\ref{Subs:Shimura} below), and give criteria of a geometric nature determining whether $f$ arises from a modular embedding. In the second group we study the behaviour of modular embeddings under abstract field automorphisms of $\bbc$, and deduce a result about the Galois operation on a certain class of dessins d'enfants.

We now announce our main results in more explicit phrasing.

\subsubsection*{Geometry of modular embeddings}
We recall the construction of irreducible arithmetic lattices in $(\PGL_2(\bbr )^+)^r$. This begins with a totally real number field $F$ and a quaternion algebra $B/F$. Let $g=[F:\bbq ]$ and let $\varrho_1,\ldots ,\varrho_g\colon F\to\bbr$ be the distinct field embeddings, and assume they are numbered in such a way that $B$ splits at $\varrho_1,\ldots ,\varrho_r$ and ramifies at $\varrho_{r+1},\ldots ,\varrho_g$. We assume that $r\ge 1$, and for every $1\le j\le r$ we fix an isomorphism $B\otimes_{F,\varrho_j}\bbr\cong\mathrm{M}_2(\bbr )$.

Let further $\mathscr{O}\subset B$ be an order and let $\mathscr{O}^1$ be the group of invertible elements of reduced norm one in~$\mathscr{O}$. By the isomorphisms we just chose this embeds as an irreducible arithmetic lattice in $\SL_2(\bbr )^r$, and any subgroup $\varGamma\subset (\PGL_2(\bbr )^+)^r$ commensurable to the image of $\mathscr{O}^1$ in $(\PGL_2(\bbr )^+)^r$ is also an irreducible arithmetic lattice.

\begin{Definition}
Let $\varGamma$ as above, let $\varDelta\subset\PGL_2(\bbr )^+$ be a lattice and let $1\le j\le r$. A \emph{modular embedding for $\varDelta$ with respect to $\varrho_j$} is a pair $(\tif ,\varphi )$ such that:
\begin{enumerate}
\item $\varphi$ is a group homomorphism $\varDelta\to\varGamma$ such that the composition
$$\varDelta\overset{\varphi }{\to }\varGamma \subset (\PGL_2(\bbr )^+)^r\overset{\mathrm{pr}_j}{\to }\PGL_2(\bbr )^+$$
is equal to the identity inclusion (in particular, $\varDelta$ is commensurable to a subgroup of the possibly indiscrete group $\varrho_j(\mathscr{O}^1)\subset\PGL_2(\bbr )^+$);
\item $\tif\colon\bbh\to\bbh^r$ is a holomorphic map equivariant for $\varphi\colon\varDelta\to\varGamma$, i.e.\ satisfying
\begin{equation*}
\tif (\delta z)=\varphi (\delta )\tif (z)\qquad\text{for all }z\in\bbh,\, \delta\in\varDelta .
\end{equation*}
\end{enumerate}
\end{Definition}
As mentioned above, such a modular embedding defines a regular map of algebraic varieties $f\colon C\to S$. Our first two main results give purely algebro-geometric (or even topological) criteria when a morphism from a curve to a quaternionic Shimura variety is covered by a modular embedding. We need to distinguish between two cases, depending on whether $C$ is affine or projective.

To state the results it is necessary to assume that $\varGamma$ is torsion-free; this can always be assumed by passing to a finite index subgroup, which is benign for our purposes. The cotangent bundle $\Omega^1_{\bbh^r}$ splits in an obvious way as a sum of line bundles, induced by the product decomposition $\bbh^r=\prod_{j=1}^r\bbh$. Since the elements of $\varGamma$ preserve this decomposition, it descends to the cotangent bundle of~$S$:
\begin{equation*}
\Omega^1_S=\bigoplus_{j=1}^r\mathcal{M}_j
\end{equation*}
for some algebraic line bundles $\mathcal{M}_j$ on~$S$.
\begin{ThrmAOne}
Let $S$ be a Shimura variety of the kind just described, constructed from a quaternion algebra $B/F$, let $C$ be a smooth projective complex curve, let $f\colon C\to S$ be a morphism of complex algebraic varieties and let $\varrho_j$ be an archimedean place of $F$ at which $B$ is unramified.
Then the following are equivalent:
\begin{enumerate}
\item $f$ is covered by a modular embedding with respect to $\varrho_j$.
\item $f^{\ast }\!\mathcal{M}_j$ is isomorphic to the canonical bundle $\omega_C$ as an algebraic line bundle.
\item $f^{\ast }\!\mathcal{M}_j$ is isomorphic to $\omega_C$ as a topological line bundle (i.e.\ the degree of $f^{\ast }\!\mathcal{M}_j$ is minus the Euler characteristic of~$C$).
\end{enumerate}
\end{ThrmAOne}
This is Corollary~\ref{CriteriaMEWithM} below. An extension of Theorem~A1, with some additional equivalent conditions, is given and proved below as Theorem~\ref{CriteriaModularEmbedding}.

Theorem~A1 is proved using simple special cases of ideas from C.~Simpson's correspondence between Higgs bundles and local systems. The line bundles $\mathcal{M}_j$ come with distinguished square roots $\mathcal{L}_j$ that can be viewed as bundles of certain modular forms for $\varGamma$, see Remark~\ref{LRhoAndModularForms} below. Phrased more abstractly, they are automorphic line bundles. For a map $f$ covered by a modular embedding, $f^{\ast }\mathcal{L}_j$ is then a theta characteristic, hence there is a natural isomorphism $\vartheta_j\colon f^{\ast }\mathcal{L}_j\to f^{\ast }\mathcal{L}_j^{-1}\otimes \omega_C^1$, i.e.\ a \emph{maximal Higgs field}. This field can also be constructed without the assumption that $f$ is covered by a modular embedding, and its maximality is equivalent to $f$ being covered by a modular embedding.

In the affine case we obtain a very similar result. Then the line bundles $f^{\ast }\mathcal{L}_j$ can be constructed in a different way. If $C$ is affine and $f\colon C\to S$ is covered by a modular embedding then $S$ is automatically a \emph{Hilbert--Blumenthal variety}, i.e.\ a quaternionic Shimura variety derived from the quaternion algebra $B=\mathrm{M}_2(F)$ (see Proposition~\ref{Prop:AffineImpliesHB} below). These varieties are moduli spaces for abelian varieties with real multiplication by $F$, hence a morphism $f\colon C\to S$ (modular embedding or not) corresponds to a family $p_C\colon A_C\to C$ of such abelian varieties. To such a family we attach its \emph{Hodge bundle} $p_{C,\ast }\Omega_{A_C/C}^1$, a vector bundle on $C$ with a $\bbq$-linear action by~$F$. Fixing an archimedean place $\varrho_j$ of $F$ we denote the $\varrho_j$-eigenspace in the Hodge bundle by $\mathcal{E}$; this is an algebraic line bundle on $C$.

Letting $\Cbar$ be a projective completion of $C$ with boundary $D$ the theory of \emph{variations of Hodge structure} provides us with a canonical extension $\mathcal{E}_{\mathrm{ext}}$ of $\mathcal{E}$ to~$\Cbar$. Our main geometric result in this case is then the following.
\begin{ThrmATwo}
Let $C=\Cbar\smallsetminus D$ be a smooth affine curve, let $p_C\colon A_C\to C$ be a family of abelian varieties with real multiplication by $F$, let $f\colon C\to S$ be the corresponding regular map where $S$ is a Hilbert--Blumenthal variety, and let $\varrho_j\colon F\to \bbr$ be an archimedean place. Let $\eta_{\mathrm{ext}}$ and $\mathcal{E}_{\mathrm{ext}}$ be as above. Then the following are equivalent:
\begin{enumerate}
\item $f$ is covered by a modular embedding for $\varrho$.
\item $\mathcal{E}_{\mathrm{ext}}$ is a logarithmic theta characteristic on~$C$, i.e.\ its square is isomorphic to $\omega_{\Cbar }(\log D)$.
\item The degree of $\mathcal{E}_{\mathrm{ext}}$ is $-\frac 12$ times the Euler characteristic of~$C$.
\end{enumerate}
\end{ThrmATwo}
This is proved as Theorem~\ref{GaussManinCriterion} below. Note the similarity to Theorem~A1: the restriction of $\mathcal{E}_{\mathrm{ext}}$ to $C$ is $f^{\ast }\mathcal{L}_j$, a square root of~$f^{\ast }\!\mathcal{M}_j$. Compare also Remark~\ref{Rem:AOneVsATwo} below.

There are other results that express a seemingly transcendental condition (as being covered by a modular embedding presents itself) on a regular map between algebraic varieties in a purely algebro-geometric language and that also proceed via Simpson's correspondence. To the present author's knowledge, this idea was first used in \cite{MR2106125} to study Shimura curves in the moduli space $\mathcal{A}_g$ of principally polarised abelian varieties, then for Teichm\"uller curves in \cite{MR2188128} and for Kobayashi geodesics in $\mathcal{A}_g$ in~\cite{MR2772554}.

\subsubsection*{Field automorphisms}
A consequence of Theorems A1 and A2 is that the property of being covered by a modular embedding is stable under field automorphisms of~$\bbc$. For $\tau\in\Aut\bbc$ the variety $\tau S$ is again a quaternionic Shimura variety and the bundle $\tau\mathcal{M}_j$ occurs again as a summand of the cotangent bundle. This implies:
\begin{ThrmB}
Let $C$ and $S$ be as in Theorem A1 or A2, let $f\colon C\to S$ be covered by a modular embedding with respect to $\varrho_j\colon F\to\bbr$ and let $\tau\in\Aut\bbc$. Then $\tau f\colon\tau C\to \tau S$ is covered by a modular embedding with respect to $\tau\circ\varrho_j$.
\end{ThrmB}
This is Theorem~\ref{ThmConjModularEmbSimple} below. A more refined version, employing the adelic formalism of Shimura varieties and thus giving a more explicit version of this result, is formulated as Theorem~\ref{ThmConjModularEmbAdelic}. It makes use of the explicit theory of Galois conjugates of Shimura varieties in the form conjectured by Langlands \cite{MR546619} and proved by Milne and Shih \cite{MR717596, MilneShih1982a}, and its extension by Milne to automorphic vector bundles~\cite{MR931206}. 

\subsubsection*{Dessins d'enfants} Theorem B (or a more precise form thereof) has applications to the theory of dessins d'enfants. Let $p,q,r$ be positive integers or $\infty$ such that
$$\frac 1p+\frac 1q+\frac 1r<1$$
(where we set $\frac 1\infty =0$), and let $\varDelta_{p,q,r}\subset\PGL_2(\bbr )^+$ be the triangle group of signature $(p,q,r)$. This group admits a modular embedding to an arithmetic group defined by a certain quaternion algebra over the field
$$K_{p,q,r}=\bbq (\tr\varDelta_{p,q,r})=\bbq\bigg(\!\cos\frac\pi p,\cos\frac\pi q,\cos\frac\pi r\bigg),$$
and we may define principal congruence subgroups $\varDelta_{p,q,r}(\mathfrak{a})$ for every integral ideal $\mathfrak{a}$ of~$K_{p,q,r}$. The natural projection
$$\varDelta_{p,q,r}(\mathfrak{a})\backslash\bbh\to\varDelta_{p,q,r}\backslash\bbh\cong\bbp^1(\bbc )$$
is then unramified outside $\{ 0,1,\infty \}\subset\bbp^1(\bbc )$, i.e.\ it is a \emph{Bely\u{\i} map} and therefore corresponds to a dessin $\mathcal{D}_{p,q,r}(\mathfrak{a})$.
\begin{ThrmC}
The action by $\Gal (\qbar /\bbq )$ on $\mathcal{D}_{p,q,r}(\mathfrak{a})$ is given by
$$\tau\mathcal{D}_{p,q,r}(\mathfrak{a})=\mathcal{D}_{p,q,r}(\tau\mathfrak{a}).$$
\end{ThrmC}
This is Theorem~\ref{Theo:ActionOnCongruenceDessins} below. Several special cases have been proved before by rather different methods, cf.\ the discussion following Theorem~\ref{Theo:ActionOnCongruenceDessins}.

\subsubsection*{Outline}

In Sections 2 and 3 we summarise known or essentially known material that is used in proving our main theorems but rather scattered throughout the literature, so we recall it here in phrasings suitable for our purposes. Section~2 contains a discussion of variations of Hodge structure, Higgs fields and Simpson's correspondence in the very special one-dimensional cases that we need, and some results about their relation to geometry, most importantly Proposition~\ref{CharacterisationMaximalHiggs}, the main nontrivial result used to prove Theorems A1 and~A2. In Section~3 we recall the formalism of connected Shimura varieties and automorphic line bundles, with special emphasis on quaternionic Shimura varieties, which are those that can occur as targets for modular embeddings.

In Section~4 we introduce modular embeddings, discuss examples and recall some of their properties, leading to proofs of Theorems A1 and~A2. The effect of field automorphisms of $\bbc$ is discussed in Section~5, beginning with a summary of Deligne--Langlands--Milne--Shih's results on conjugates of Shimura varieties and automorphic bundles and continuing with applications to modular embeddings, including Theorems B and~C.

\subsubsection*{Acknowledgements}

This article started as a chapter in \cite{KucharczykVoight2015}. The author wishes to thank his coauthor, John Voight, for suggesting to publish it as a separate article, and J\"urgen Wolfart for encouraging him to try to prove Theorem~C using modular embeddings. Furthermore he wishes to thank these two people and in addition Ursula Hamenst\"adt, Martin M\"oller, Joachim Schwermer, Jakob Stix, Martin Westerholt-Raum and Jonathan Zachhuber for insightful conversations about the contents of this article and surrounding topics. Last but not least thanks go to Allan Hale for printing.

\subsubsection*{Notation}

For a field $k$ the symbol $\widebar{k}$ denotes an algebraic closure of~$k$. Specifically, $\qbar$ denotes \emph{the} algebraic closure of $\bbq$ in~$\bbc$. Complex conjugation, or its restriction to any subfield of $\bbc$, is denoted by~$\iota$. The profinite completion of $\bbz$ is denoted by $\hat{\bbz }=\prod_{p\text{ prime}}\bbz_p$ and the ring of finite rational ad\`{e}les is denoted by $\bba^f=\hat{\bbz }\otimes_{\bbz }\bbq$.

For a field extension $E/F$, Weil's restriction of scalars is denoted by $\Res_{E/F}$. The multiplicative group over a field $F$ is denoted by $\bbg_m=\bbg_{m,F}=\GL_{1,F}$. We often write $\bbt^E$ for $\Res_{E/F}\bbg_m$ when $F$ is clear from the context. For a linear algebraic group $G$ over a field $k$ we write $\mathrm{X}^{\ast }(G)=\Hom (G_{\widebar{k} },\bbg_{m,\widebar{k} })$ and $\mathrm{X}_{\ast }(G)=\Hom (\bbg_{m,\widebar{k} },G_{\widebar{k} })$, both regarded as $\Gal (\widebar{k}/k)$-modules.

For an algebraic group $G$ defined over a subfield $k$ of $\bbr$ we denote the connected component of the identity in $G(\bbr )$ by $G(\bbr )^+$, and we set $G(k)^+=G(k)\cap G(\bbr )^+$.

The group of $n$-th roots of unity (as an abstract group scheme, or as a subgroup of $\bbc$) is denoted by~$\mu_n$.

Holomorphic or algebraic vector bundles on a complex manifold or variety $X$ are denoted by latin letters such as $V$ when considered as fibre bundles $V\to X$ with fibrewise vector space structure, and by corresponding calligraphic letters such as $\mathcal{V}$ when considered as sheaves of $\mathcal{O}_X$-modules.

If $X$ is a variety over $\bbc$, the associated complex analytic space is denoted by $X^{\an }$.

Two different symbols for isomorphism are used: $X\simeq Y$ denotes the existence of an isomorphism $X\to Y$ and $X\cong Y$ denotes that a particular (usually canonical) isomorphism has been chosen.

If $f\colon X\to Y$ is a map and $A\subseteq X$ and $B\subseteq Y$ are subsets with $f(A)\subseteq B$, then $f|_A^B$ denotes the induced map $A\to B$, and $f|A$ means $f|_A^A$ when it makes sense.

For two sets $A$ and $B$ the symmetric difference is denoted by $A\triangle B=(A\cup B)\smallsetminus (A\cap B)$.

For an equivalence relation on a set $X$ the equivalence class containing $x\in X$ is denoted by $[x]$; for equivalence classes in products $X\times Y$ we write $[x,y]$ instead of $[(x,y)]$.

\section{Uniformisation and Higgs fields}

\noindent By the uniformisation theorem the universal cover $\tilde{X}$ of every finite type Riemann surface $X$ with negative Euler characteristic is biholomorphic to the upper half plane $\bbh=\{ z\in\bbc\mid\operatorname{Im} z>0\}$, hence $X$ is biholomorphic to a quotient $\varDelta\backslash\bbh$ where $\varDelta$ is a torsion-free lattice in the biholomorphism group $\Aut\bbh$. We identify this group with $\PGL_2(\bbr )^+$ acting via M\"obius transformations.\footnote{This is more commonly called $\operatorname{PSL}_2(\bbr )$, but we wish to avoid this notation since it wrongly suggests the existence of a linear algebraic group $\operatorname{PSL}_2$ whose group of real-valued points is $\PGL_2(\bbr )^+$.} The geometry behind this result is much subtler than it appears at first sight, see e.g.~\cite{MR700005, MR0207978}.

The upper half plane naturally occurs in many different contexts. For the purpose of the present work it is convenient to interpret it as a \emph{period domain for Hodge structures}.

\subsection{The upper half plane and variations of Hodge structure}\label{Sect:HodgeStructures}

We recall the basic definitions around Hodge structures, cf.~\cite{MR0498551, MR0282990, MR0382272}. Let $V$ be a finite-dimensional real vector space. A \emph{real Hodge structure} on $V$ is a decomposition of $V_{\bbc }=V\otimes_{\bbr }\bbc$ into complex vector subspaces,
$$V_{\bbc }=\bigoplus_{p,q\in\bbz }V^{p,q},$$
such that $V^{p,q}$ is the complex conjugate of $V^{q,p}$. The \emph{Hodge numbers} of this structure are the numbers $v^{p,q}=\dim_{\bbc }V^{p,q}$. There are two important data derived from this decomposition:
\begin{enumerate}
\item the \emph{weight decomposition} $V=\bigoplus_{n\in\bbz }V_n$ (direct sum of real vector spaces), where $V_{n,\bbc }=\bigoplus_{p+q=n}V^{p,q}$, and
\item the \emph{Hodge filtration} $\mathrm{F}^{\bullet }V_{\bbc }\colon\cdots\supseteq\mathrm{F}^{-1}V_{\bbc }\supseteq\mathrm{F}^0V_{\bbc }\supseteq\mathrm{F}^1V_{\bbc }\supseteq\cdots$ defined by $\mathrm{F}^pV_{\bbc }=\bigoplus_{p'\ge p}V^{p',q}$.
\end{enumerate}
The weight decomposition and the Hodge filtration together suffice to reconstruct the Hodge structure as
$$V^{p,q}=V_{\bbc ,p+q}\cap\mathrm{F}^pV_{\bbc }\cap\overline{\mathrm{F}^qV_{\bbc }}.$$
A Hodge structure is called \emph{pure of weight $n$} if $V=V_n$. The standard example of a Hodge structure of pure weight $n$ is $\mathrm{H}^n(X,\bbr )$ for a compact K\"ahler manifold~$X$.

These definitions can be rephrased in terms of representations of the \emph{Serre torus} $\bbs =\Res_{\bbc /\bbr }\bbg_m$ with $\bbs (\bbr )=\bbc^{\times }$. If $V$ is a finite-dimensional real vector space and $\bbs\to\GL (V)$ is a rational representation, we let $V^{p,q}$ be the subspace of $V_ {\bbc }$ where $z\in\bbc^{\times }=\bbs (\bbr )$ operates as multiplication by $z^p\widebar{z}^q$. This assignment defines an equivalence of categories between representations of $\bbs$ and real Hodge structures.

The weight decomposition corresponds to the homomorphism $w\colon\bbg_{m,\bbr }\to\bbs$ given on real-valued points by $r\mapsto r^{-1}$. Hence the weight decomposition is characterised by the fact that $w(r)$ operates as $r^{-n}$ on~$V_n$.

The Hodge filtration corresponds to  the homomorphism $\mu\colon\bbg_{m,\bbc }\to\bbs_{\bbc }$ given on complex-valued points by $\bbc^{\times }\to\bbc^{\times }\times\bbc^{\times }$, $z\mapsto (z,1)$. Here we use the unique isomorphism of complex Lie groups $\bbs (\bbc )\to\bbc^{\times }\times\bbc^{\times }$ such that the composition
$$\bbc^{\times }=\bbs (\bbr )\hookrightarrow\bbs (\bbc )\cong\bbc^{\times }\times\bbc^{\times }$$
is the identity on the first coordinate and complex conjugation on the second one.

For this reason, whenever we have a homomorphism $h\colon\bbs\to G$ where $G$ is a real algebraic group, we think of it as a `Hodge structure on $G$' and call $w_h=h\circ w\colon\bbg_{m,\bbr }\to G$ the `weight decomposition' and $\mu_h=h\circ\mu\colon\bbg_{m,\bbc }\to G$ the `Hodge filtration'.

Tensor products, duals, symmetric powers and alternating powers of Hodge structures are defined in the obvious ways. For each $n\in\bbz$ we define the \emph{Tate Hodge structure}\footnote{Not to be confused with Hodge--Tate structures, which are certain $p$-adic Galois representations.} $\bbr (n)$ of pure weight $-2n$ that corresponds to the representation $\bbs\to\GL (1)$ sending $z\in\bbc^{\times }=\bbs (\bbr )$ to~$\lvert z\rvert^{-n}$. Its underlying real vector space is usually written as $(2\pi\mathrm{i})^n\bbr$, and for nonnegative $n$ there is a canonical isomorphism $\bbr (-n)\cong\Hup^{2n}(\bbp^n(\bbc ),\bbr )$.

\subsubsection*{Polarisations}

For a real Hodge structure $V$ with corresponding representation $\varrho\colon\bbs\to\GL (V)$ the \emph{Weil operator} on $V$ is the real linear map $C\colon V\to V$ that acts as $\mathrm{i}^{p-q}$ on $V^{p,q}$. It can also be described as $\varrho (\mathrm{i})$ where $\mathrm{i}$ is interpreted as an element of $\bbc^{\times }=\bbs (\bbr )$.

Let $V$ be a real Hodge structure which is pure of weight~$n$. A \emph{polarisation} of $V$ is a morphism of Hodge structures $\psi\colon \bigwedge^2V\to\bbr (-n)$ for which the symmetric pairing
$$\Psi\colon V_{\bbr }\times V_{\bbr }\to\bbr ,\qquad (v,w)\mapsto (2\pi\mathrm{i})^n\psi (v,Cw),$$
is positive definite. The Hodge structures $\Hup^n(X^{\an },\bbr )$ for smooth projective varieties $X$ admit polarisations coming from ample line bundles on~$X$.

\subsubsection*{Variations}

If $X$ is a complex manifold, a \emph{variation of (real) Hodge structure} $\bbv$ on $X$ consists of a local system $\bbv_{\bbr }$ of real vector spaces on $X$ together with a Hodge structure on each fibre $\bbv_{\bbr }(x)$ such that the weight decomposition is induced by a decomposition of local systems $\bbv_{\bbr }=\bigoplus_{n\in\bbz }\bbv_{\bbr ,n}$, the Hodge filtration is induced by a filtration of holomorphic vector bundles $\mathrm{F}^{\bullet }\mathcal{V}\colon\cdots\supseteq\mathrm{F}^{-1}\mathcal{V}\supseteq\mathrm{F}^0\mathcal{V}\supseteq\mathrm{F}^1\mathcal{V}\supseteq\cdots$ and Griffiths's transversality condition $\nabla\mathrm{F}^p\mathcal{V}\subseteq\mathrm{F}^{p-1}\mathcal{V}\otimes\Omega_X^1$ holds. Here $\mathcal{V}=\bbv_{\bbr }\otimes_{\bbr }\mathcal{O}_X$ is the holomorphic vector bundle underlying $\bbv_{\bbc }$ and $\nabla\colon\mathcal{V}\to\mathcal{V}\otimes\Omega_X^1$ is the unique holomorphic integrable connection for which the sections of $\bbv_{\bbc }$ are flat.

Morphisms, tensor products and so forth are defined in the obvious way, as is the `constant variation' $\underline{V}_X$ associated with a Hodge structure~$V$. A \emph{polarisation} of a variation $\bbv$ of pure weight $n$ is a morphism of variations $\bigwedge^2\bbv\to\underline{\bbr (-n)}_X$ which is a polarisation fibrewise. Usually the existence of a polarisation on each weight component is included in the definition of a variation of Hodge structure.

Again the motivating example is derived from the cohomology of algebraic varieties: let $f\colon Y\to X$ be a smooth projective morphism of algebraic varieties, then there is a variation of real Hodge structure $\bbv$ with underlying local system $\bbv_{\bbr }=\mathrm{R}^nf_{\ast }\bbr$; the fibre at $x\in X$ is $\Hup^n(f^{-1}(x)^{\an },\bbr )$ equipped with its natural Hodge structure. These variations admit polarisations that are constructed from relatively ample line bundles on~$Y$.

\subsubsection*{The upper half plane as a period domain}

Let $W=\bbr^2$ with the symplectic form
\begin{equation}\label{StandardPolarisation}
\psi\colon W\times W\to\bbr (-1)=\frac{1}{2\pi\mathrm{i}}\bbr,\qquad \left({v_1\choose v_2}, {w_1\choose w_2}\right)\mapsto \frac{1}{2\pi\mathrm{i}}(v_1w_2-v_2w_1).
\end{equation}
Then $\bbh$ parameterises Hodge structures on $W$ for which $\psi$ is a polarisation and which have Hodge numbers $v^{1,0}=v^{0,1}=1$ and all other $v^{p,q}=0$. An easy calculation shows that all such Hodge structures have the form
$$\mathrm{F}^1W_{\bbc }=W^{1,0}=\bbc\cdot {z\choose 1}\quad\text{and}\quad W^{0,1}=\bbc\cdot {\widebar{z}\choose 1}$$
for a uniquely determined $z\in\bbh$.

These Hodge structures define a polarised variation of Hodge structure $\bbw$ on~$\bbh$, given by the following data:
\begin{enumerate}
\item the trivial local system of $\bbr$-vector spaces $\bbw_{\bbr }=\underline{W}$ on $\bbh$;
\item the symplectic form $\psi\colon\bigwedge^2\bbw_{\bbr }\to\underline{\bbr (-1)}$;
\item the filtration of the trivial holomorphic vector bundle $\mathcal{W}=\bbw_{\bbr }\otimes_{\bbr }\mathcal{O}_{\bbh}\cong\mathcal{O}_{\bbh}^2$ with $\mathrm{F}^p\mathcal{W}=\mathcal{W}$ for $p\le 0$,
$$\mathrm{F}^1\mathcal{W}\subset\mathcal{W},\qquad (\mathrm{F}^1\mathcal{W})_z=\bbc \cdot {z\choose 1},$$
and $\mathrm{F}^0\mathcal{W}=0$ for $p\ge 2$.
\end{enumerate}
The previous considerations show:
\begin{Lemma}
Let $Y$ be a complex manifold. Then there is a natural bijection between the following two types of objects:
\begin{enumerate}
\item polarised variations of real Hodge structure $\bbv$ on the trivial local system $\bbv_{\bbr }=\underline{\bbr^2}_{Y}$ with the symplectic form (\ref{StandardPolarisation}) with Hodge numbers $v^{1,0}=v^{0,1}=1$ and all other $v^{p,q}=0$;
\item holomorphic maps $f\colon Y\to\bbh$.
\end{enumerate}
This bijection is effectuated as follows: $f$ as in (ii) corresponds to the variation $f^{\ast }\bbw$ on $Y$, and a variation $\bbv$ corresponds to the \emph{period map}
$$f\colon Y\to \bbh,\qquad y\mapsto z\text{ such that }(\mathrm{F}^1\mathcal{V})_y=\bbc\cdot {z\choose 1}.$$
\end{Lemma}

The universal variation $\bbw$ on $\bbh$ admits an action by $\SL_2(\bbr )$ extending that on $\bbh$ by M\"{o}bius transformations: if $W=\bbh\times\bbr^2$ is the total space of $\bbw_{\bbr }$, we set
$$g\colon W\to W,\qquad (z,w)\mapsto (gz,gw),$$
for $g\in\SL_2(\bbr )$, where $z\mapsto gz$ is the usual operation by M\"{o}bius transformations and $w\mapsto gw$ is the usual matrix-by-vector multiplication. A direct computation shows that the action of $\SL_2(\bbr )$ on $\bbw_{\bbc }$ by $\bbc$-linear extension preserves the Hodge filtrations, and therefore this action can be viewed as an action on $\bbw$ as a polarised variation of real Hodge structure.

If $\varGamma\subset\SL_2(\bbr )$ is a torsion-free discrete subgroup, then $\bbw$ descends to a polarised variation of real Hodge structure $\bbw_{\varGamma }=\varGamma\backslash\bbw$ on $X_{\varGamma }=\varGamma\backslash\bbh$ which we call the \emph{Fuchsian variation} on $X_{\varGamma }$.

\begin{Remark}
If $\varGamma\subset\SL_2(A)$ for some subring $A\subseteq\bbr$, the Fuchsian variation admits a natural $A$-structure $\bbw_{\varGamma ,A}$. In particular for torsion-free subgroups $\varGamma\subset\SL_2(\bbz )$ we obtain a principally polarised variation of \emph{integral} Hodge structure on $X_{\varGamma }$; it is isomorphic to the variation $\mathrm{R}^1p_{\ast }\bbz$ for the universal or `modular' family of elliptic curves $p\colon E_{\varGamma }\to X_{\varGamma }$, and the family can be reconstructed from the variation. In this sense the Fuchsian variation $\bbw_{\varGamma }$ serves as a replacement for this family of elliptic curves also for the cases when $\varGamma$ is not contained in $\SL_2(\bbz )$.
\end{Remark}

\begin{Definition}
Let $X$ be a complex manifold. A \emph{real variation of elliptic curve type} on $X$ is a polarised variation of real Hodge structure $\bbv$ on $X$ with Hodge numbers $v^{1,0}=v^{0,1}=1$ and all other $v^{p,q}=0$.
\end{Definition}
For such a variation the local system $\bbv_{\bbr }$ need not be trivial. However, its pullback to the universal cover $\tilde{X}$ is. After choosing a trivialisation of this pullback we obtain a period map $\tilde{X}\to\bbh$. This trivialisation allows us to interprete the monodromy of the local system on $X$ as a representation $\varrho\colon\pi_1(X)\to\SL_2(\bbr )$; the period map is equivariant for~$\varrho$. If the image of $\varrho$ is a \emph{discrete} subgroup $\varGamma$ of $\SL_2(\bbr )$, then the period map descends to a holomorphic map $f\colon X\to \varGamma\backslash\bbh$, and $\bbv\cong f^{\ast }\bbw_{\varGamma }$.

\subsubsection*{Deligne's canonical extension} Let $\widebar{X}$ be a smooth compact complex manifold and let $Y\subset\widebar{X}$ be a normal crossings divisor (these conditions are met if $\widebar{X}$ is a compact Riemann surface and $Y$ is a finite subset). Set $X=\widebar{X}\smallsetminus Y$, and let $\bbv$ be a polarised variation of real Hodge structure on~$X$. Assume furthermore that the monodromy around $Y$ operates by unipotent endomorphisms on the fibres of~$\bbv_{\bbr }$. This is true, for instance, if $X=\varGamma\backslash\bbh$ and $\bbv =\bbw_{\varGamma }$ for a torsion-free lattice $\varGamma\subset\SL_2(\bbr )$ that does not contain parabolic elements with eigenvalue $-1$.  Let $\mathcal{V}=\bbv_{\bbr }\otimes_{\bbc }\mathcal{O}_X$ be the vector bundle underlying the variation, and let $\nabla\colon\mathcal{V}\to\mathcal{V}\otimes\Omega_X^1$ be the natural connection.

Then there exist an extension of $\mathcal{V}$ to a holomorphic vector bundle $\mathcal{V}_{\mathrm{ext}}$ on $\widebar{X}$ and an extension of $\nabla$ to a logarithmic connection, i.e.\ a sheaf homomorphism $\nabla_{\!\mathrm{ext}}\colon\mathcal{V}_{\mathrm{ext}}\to\mathcal{V}_{\mathrm{ext}}\otimes\Omega_{\widebar{X}}^1(\log Y)$ satisfying the Leibniz rule, such that the residue of $\nabla_{\!\mathrm{ext}}$ at each component $Y^0$ of $Y$ is a nilpotent endomorphism of $\mathcal{V}_{\mathrm{ext}}|_{Y^0}$. The pair $(\mathcal{V}_{\mathrm{ext}},\nabla_{\!\mathrm{ext}})$ is determined uniquely up to isomorphism by these properties. Its existence and uniqueness are proved in \cite[II.\S 5]{MR0417174}.

The subbundles $\mathrm{F}^p\mathcal{V}\subseteq\mathcal{V}$ extend uniquely to holomorphic subbundles $\mathrm{F}^p\mathcal{V}_{\mathrm{ext}}\subseteq\mathcal{V}_{\mathrm{ext}}$. This is a direct consequence of the Nilpotent Orbit Theorem \cite[(4.9)]{MR0382272}. 

\begin{Remark}\label{AlgebraicStructuresOnHodgeSubbundles}
Assume that $\widebar{X}$ is a smooth projective variety, and that $\bbv$ is isomorphic to the variation $\mathrm{R}^nf_{\ast }\bbr$ for some smooth proper morphism $f\colon Y\to X$. Then the vector bundle $\mathcal{V}$ has an algebraic structure: it can be identified with the relative de Rham cohomology $\mathcal{H}^1_{\mathrm{dR}}(Y/X)=\mathrm{R}^1f_{\ast }\Omega_{Y/X}^{\bullet }$, and the subbundles $\mathrm{F}^p\mathcal{V}$ can be constructed by truncating the relative de Rham complex $\Omega_{Y/X}^{\bullet }$; hence they, too, are algebraic subbundles of $\mathcal{V}\!$. But there is also a different way to find algebraic structures on these bundles: the canonical extensions $\mathcal{V}_{\mathrm{ext}}$ and $\mathrm{F}^p\mathcal{V}_{\mathrm{ext}}$ are holomorphic vector bundles on $\widebar{X}$, hence by Serre's `GAGA' theorems \cite[Th\'eor\`emes 2 \'et~3]{MR0082175} they possess unique algebraic structures which we can restrict to~$X$. By \cite[(4.13)]{MR0382272} these two constructions produce the same algebraic structures on $\mathcal{V}$ and its Hodge subbundles.
\end{Remark}

\subsection{Higgs bundles}

It is often convenient to work with systems of Hodge bundles (see \cite[p.~898]{MR944577}), a souped-up version of Higgs bundles, instead of variations of Hodge structure, since they fit more naturally into the realm of algebraic varieties. By deep results of C.\ Simpson \cite{MR944577, MR1040197, MR1179076}, variations of Hodge structure and systems of Hodge bundles contain the same information in many important cases, including our setup. For an introduction to this circle of ideas see~\cite{MR1157844}. We need only an elementary special case, for which the whole machinery is not necessary.

\begin{Definition}
\begin{enumerate}
\item Let $X$ be a complex manifold or an algebraic variety (over an arbitrary base field). A \emph{Higgs bundle of elliptic curve type} on $X$ consists of a (holomorphic or algebraic) line bundle $\mathcal{L}$ on $X$ and an $\mathcal{O}_X$-linear map $\vartheta\colon\mathcal{L}\to\mathcal{L}^{-1}\otimes_{\mathcal{O}_X}\Omega_X^1$, called the \emph{Higgs field}.
\item If $\widebar{X}$ is a smooth compact complex manifold or a smooth complete algebraic variety and $Y\subset \widebar{X}$ is a normal crossings divisor, a \emph{logarithmic Higgs bundle of elliptic curve type} on $X=\widebar{X}\smallsetminus Y$ is a line bundle $\mathcal{L}$ on $\widebar{X}$ together with an $\mathcal{O}_{\widebar{X}}$-linear map $\vartheta\colon\mathcal{L}\to\mathcal{L}^{-1}\otimes\Omega_{\widebar{X}}^1(\log Y)$.
\end{enumerate}
\end{Definition}
Let $X$ be a complex manifold and let $\bbv$ be a real variation of elliptic curve type on~$X$. There is a simple way to construct a Higgs bundle from~$\bbv$: Set $\mathcal{L}=\mathrm{F}^1\mathcal{V}$ and (for the moment) $\mathcal{M}=\mathcal{V}/\mathcal{L}$. The fibre of $\mathcal{M}$ over $x\in X$ is $\bbv_{x,\bbc }/\bbv_x^{1,0}$ which is canonically isomorphic to $\bbv_x^{0,1}$; note, however, that the $\bbv_x^{0,1}$ for varying $x$ form not a holomorphic but an antiholomorphic subbundle of $\bbv_{\bbc }$, which is isomorphic as a $\mathcal{C}^{\infty }$-vector bundle to $\mathcal{M}\otimes_{\mathcal{O}_Y}\mathcal{C}_Y^{\infty }$.\label{HiggsBundlesParagraph}
\begin{Lemma}\label{HiggsBundleSelfDualForVHS}
With this notation, $\mathcal{M}$ is canonically isomorphic to the dual of~$\mathcal{L}$.
\end{Lemma}
Hence we may identify $\mathcal{M}$ with~$\mathcal{L}^{-1}$.
\begin{proof}
The short exact sequence
$$0\to\mathcal{L}\to\mathcal{V}\to\mathcal{M}\to 0$$
gives rise to an isomorphism $\mathcal{L}\otimes\mathcal{M}\cong\bigwedge^2\mathcal{V}\!$. But $\bigwedge^2\mathcal{V}$ is the vector bundle underlying the variation of Hodge structure $\bigwedge^2\bbv$, which is canonically isomorphic to the constant variation $\underline{\bbr (-1)}$ via the polarisation.
\end{proof}
Recall that the connection $\nabla\colon\mathcal{V}\to\mathcal{V}\otimes\Omega_X^1$ is $\bbc$-linear and satisfies the Leibniz rule:
\begin{equation}\label{LeibnizRuleConnection}
\nabla (fv)=f\nabla (v)+v\otimes\mathrm{d}f
\end{equation}
for sections $f\in\mathcal{O}_X(U)$, $v\in\mathcal{V}(U)$. We then define a Higgs field $\vartheta\colon\mathcal{L}\to\mathcal{L}^{-1}\otimes_{\mathcal{O}_X}\Omega_X^1$ as the composition of inclusion, connection and projection:
\begin{equation*}
\vartheta\colon\mathcal{L}\hookrightarrow\mathcal{V}\overset{\nabla }{\to }\mathcal{V}\otimes\Omega_{X}^1\to\mathcal{M}\otimes\Omega_{X}^1\cong\mathcal{L}^{-1}\otimes\Omega_X^1.
\end{equation*}
This is indeed an $\mathcal{O}_X$-linear homomorphism of sheaves: if $v$ is a section of $\mathcal{L}$, then the second summand in (\ref{LeibnizRuleConnection}) vanishes modulo $\mathcal{L}\otimes\Omega_X^1$.

Similarly, if $\widebar{X}$ is a smooth compact complex manifold, $Y\subset \widebar{X}$ is a normal crossings divisor, $X=\widebar{X}\smallsetminus Y$ and $\bbv$ is a real variation of elliptic curve type on $X$ with unipotent monodromy around $Y$, then the Deligne extension of the corresponding vector bundles gives rise to a logarithmic Higgs bundle $\vartheta\colon\mathcal{L}\to\mathcal{L}^{-1}\otimes\Omega_{\widebar{X}}^1(\log Y)$. If $\widebar{X}$ is a complex projective variety then $\mathcal{L}$ and $\vartheta$ are again algebraic by `GAGA'.

\subsection{Theta characteristics}\label{Subs:Theta}

Theta characteristics are structures on algebraic curves or Riemann surfaces that can be expressed in many different ways. They also play a r\^{o}le in our theory, in the following two guises:
\begin{enumerate}
\item Let $C$ be a smooth curve over a field $k$, let $\Cbar$ be its smooth projective completion and let $D=\Cbar\smallsetminus C$ be the boundary. A \emph{(logarithmic) theta characteristic} on $C$ is an algebraic line bundle $\mathcal{L}$ on $\Cbar$ such that $\mathcal{L}^{\otimes 2}\simeq\omega_{\Cbar }(\log D)$. Let $\Theta (C)$ be the set of isomorphism classes of logarithmic theta characteristics on~$C$. 

Assume from now on that $k$ is algebraically closed of characteristic different from two. From the structure of the Picard group $\Pic \Cbar$ one easily derives that $\Theta (C)$ is nonempty if and only if the cardinality of $D$ is even, and in this case it has $2^{2g}$ elements where $g$ is the genus of~$\Cbar$. More precisely, consider the long exact cohomology sequence for the short exact sequence
\begin{equation}\label{SesForThetaCharacteristicsEtaleCohomology}
1\to\mu_2\to\mathcal{O}_{\Cbar }^{\times }\overset{(\cdot )^2}{\longrightarrow }\mathcal{O}_{\Cbar }^{\times }\to 1
\end{equation}
of \'etale sheaves on $\Cbar$; a segment of it can be rewritten as
\begin{equation*}
\cdots k^{\times }\overset{(\cdot )^2}{\to }k^{\times }\overset{0}{\to }\Hup^1_{\et}(\Cbar,\mu_2)\to\Pic \Cbar\overset{2\cdot }{\to }\Pic \Cbar\to\Hup^2_{\et }(\Cbar ,\mu_2)\to 0
\end{equation*}
(cf.\ the discussion in \cite[III.3]{MR0463174}). So, $\Hup^1_{\et }(\Cbar ,\mu_2)$, which is a $(\bbz /2\bbz )$-vector space of rank $2g$, is identified with the $2$-torsion subgroup of $\Pic \Cbar$, and hence $\Theta (C)$ is, when nonempty, a torsor under $\Hup_{\et }^1(\Cbar ,\mu_2)$.

The same discussion can be repeated for Riemann surfaces of finite type, with holomorphic line bundles instead of algebraic line bundles and sheaves for the analytic topology instead of the \'{e}tale topology.
\item Let $\varDelta\subset\PGL_2(\bbr )^+\cong\SL_2(\bbr )/\{\pm 1\}$ be a torsion-free lattice. Then a \emph{theta characteristic for $\varDelta$} is a lift of $\varDelta$ to $\SL_2(\bbr )$, i.e.\ a group homomorphism $\varDelta\to\SL_2(\bbr )$ such that the composition $\varDelta\to\SL_2(\bbr )\to\PGL_2(\bbr )^+$ is the given embedding. A theta characteristic of $\varDelta$ is called \emph{unipotent} if all parabolic elements of $\varDelta$ are lifted to unipotent elements of $\SL_2(\bbr )$. Denote the set of theta characteristics of $\varDelta$ by $\Theta (\varDelta )$, and the subset of unipotent ones by $\Theta^{\mathrm{uni}}(\varDelta )\subseteq\Theta (\varDelta )$.

Again it is a classical result, rediscovered several times, that there always exists a theta characteristic for~$\varDelta$: see \cite{MR825087, MR700005, MR780044, MR1124819}. A unipotent theta characteristic exists if and only if the number of cusps for $\varDelta$ is even, cf.\ the discussion preceding \cite[Theorem~4]{MR700005}. If $\varrho\colon\varDelta\to\SL_2(\bbr )$ is a theta characteristic, then every other theta characteristic of $\varDelta$ can be obtained as $\varrho\cdot\varepsilon$, where $\varepsilon$ runs through all group homomorphisms $\varDelta\to\mu_2=\{\pm 1\}$. Hence $\Theta (\varDelta )$ is a torsor under the group $\Hom (\varDelta, \mu_2)=\Hup^1(\varDelta ,\mu_2)$. Similarly $\Theta^{\mathrm{uni}}(\varDelta )$, if nonempty, is a torsor under the group $\Hup^1_{\mathrm{cusp}}(\varDelta ,\mu_2)$ of all group homomorphisms $\varDelta\to\mu_2$ sending all parabolic elements to~$1$.
\end{enumerate}
It is no coincidence that these two sets have such similar cohomological descriptions: if $C$ is a smooth curve over $\bbc$ and $\varDelta =\pi_1(C^{\an })$ is viewed as a lattice in $\PGL_2(\bbr )^+$ via the uniformisation theorem, then there is a natural bijection $\Theta^{\mathrm{uni}} (\varDelta )\to\Theta (C)$. It is equivariant for the isomorphism $\Hup^1_{\mathrm{cusp}}(\varDelta ,\mu_2)\cong\Hup^1_{\et }(\Cbar ,\mu_2)$ obtained from the identification of $\pi_1(\Cbar^{\an })$ with the quotient of $\varDelta$ by the normal subgroup generated by its parabolic elements.

To construct it, let $\varrho\colon\varDelta\to\SL_2(\bbr )$ be an element of $\Theta^{\mathrm{uni}} (\varDelta )$, and let $\varGamma =\varrho (\varDelta )$. Then $C^{\an }\cong\varGamma\backslash\bbh$, and so we can view the Fuchsian variation $\bbw_{\varGamma }$ as a variation on~$C^{\an }$. By construction it has unipotent monodromy around the cusps. Therefore we obtain a Higgs field $\vartheta\colon\mathcal{L}\to\mathcal{L}^{-1}\otimes\Omega_{\widebar{C}}^1(\log D)$.
\begin{Proposition}
The Higgs field $\vartheta\colon\mathcal{L}\to\mathcal{L}^{-1}\otimes\omega_{\Cbar }(\log D)$ is an isomorphism, hence $\mathcal{L}$ is a logarithmic theta characteristic on~$C$.
\end{Proposition}
\begin{proof}
The Higgs field corresponding to the universal variation $\bbw$ on $\bbh$ is easily seen to be an isomorphism, hence so is the restriction of $\vartheta$ to~$C$. It remains to check isomorphy at the cusps. As a local model consider the quotient variation $\bbw_{\varGamma_{\infty }}$ on $\varGamma_{\infty }\backslash\bbh\cong\mathbb{D}\smallsetminus\{ 0\}$ for $\varGamma_{\infty }$ the cyclic group generated by $\begin{pmatrix}1& 1\\ 0& 1\end{pmatrix}$. Its underlying vector bundle is isomorphic to $\mathcal{V}=\mathcal{O}^2$ on $\mathbb{D}\smallsetminus\{ 0\}$ with Hodge filtration $\mathrm{F}^1\mathcal{V}=0\oplus\mathcal{O}$ and connection
\begin{equation*}
\nabla =\mathrm{d}-\begin{pmatrix}0&\frac{\mathrm{d}t}{t}\\ 0&0\end{pmatrix}.
\end{equation*}
From this description one directly reads off the Deligne extension of $\bbw_{\varGamma_{\infty }}$ to $\mathbb{D}$ and its corresponding Higgs field, which is given by
$$\mathcal{O}_{\mathbb{D}}\to\mathcal{O}_{\mathbb{D}}\otimes\omega_{\mathbb{D}}(0),\quad f(t)\mapsto \frac{f(t)\,\mathrm{d}t}{t},$$
hence an isomorphism.
\end{proof}
The map $\varGamma\mapsto [\mathcal{L}]$ is the desired bijection $\Theta^{\mathrm{uni}} (\varDelta )\to\Theta (C)$. Its equivariance for the natural isomorphism $\Hup^1_{\mathrm{cusp}}(\varDelta ,\mu_2)\to\Hup^1_{\et }(\Cbar ,\mu_2)$ can be checked by an elementary calculation involving automorphy factors, similar to the proof of \cite[Theorem~1]{MR700005}.
\begin{Proposition}\label{CharacterisationMaximalHiggs}
Let $\Cbar$ be a smooth projective curve over~$\bbc$, let $D\subset \Cbar (\bbc )$ be a finite subset and let $C=\Cbar\smallsetminus D$. Assume that $C$ is of hyperbolic type, i.e.\ the universal covering of $C$ is~$\bbh$. Let $\bbv$ be a real variation of elliptic curve type on $C^{\an }$ which has unipotent monodromy around all cusps and let $\vartheta\colon\mathcal{L}\to\mathcal{L}^{-1}\otimes\omega_{\Cbar }(\log D)$ be the associated Higgs field on~$\Cbar$.

 Then the following are equivalent:
\begin{enumerate}
\item The monodromy representation $\pi_1(C^{\an })\to\SL_2(\bbr )$ of $\bbv_{\bbr }$ is conjugate to a unipotent theta characteristic of~$\pi_1(C^{\an })$.
\item The period map $p\colon \widetilde{C^{\an }}\to\bbh$ is a biholomorphism.
\item $\vartheta$ is an isomorphism of line bundles (`the Higgs field is maximal').
\item $\mathcal{L}$ is a logarithmic theta characteristic on~$C$.
\item The degree of the line bundle $\mathcal{L}$ is $-\frac{1}{2}\chi (C)$, where $\chi$ denotes the topological Euler characteristic.
\end{enumerate}
If these conditions are satisfied, then the unipotent theta characteristic of $\pi_1(C^{\an })$ in (i) and the theta characteristic of $C$ in (iv) correspond to each other under the bijection $\Theta^{\mathrm{uni}}(\pi_1(C^{\an }))\leftrightarrow\Theta (C)$.
\end{Proposition}
\begin{proof}
The equivalence (i) $\Leftrightarrow$ (ii) is an elementary consequence of the fact that the period map is equivariant for the monodromy representation, and (ii) $\Leftrightarrow$ (iii) is \cite[Lemma~2.1]{MR2106125}. A more abstract version of this argument can be found in the discussion preceding \cite[Prop.~9.1]{MR944577}).

The implications (iii) $\Rightarrow$ (iv) $\Rightarrow$ (v) are obvious. For the remaining part (v) $\Rightarrow$ (iii) note that if (v) is true then $\vartheta$ cannot be identically zero: assume it is. Then by construction the Hodge decomposition is flat, i.e.\ $\mathcal{L}=\mathrm{F}^1\mathcal{V}$ is the line bundle underlying a local subsystem of rank one in~$\bbv_{\bbc }$. But such a line bundle has degree zero, contradiction. Hence the Higgs field is a nontrivial holomorphic section of the degree zero line bundle $\underline{\Hom }(\mathcal{L},\mathcal{L}^{-1}\otimes\omega_{\Cbar }(\log D))$; but as the degree is the number of zeros minus the number of poles and $\vartheta$ has no poles, it cannot have any zeros either, hence it is invertible.
\end{proof}

\section{Quaternionic Shimura varieties}

\noindent We now review the construction and some important properties of the Shimura varieties that occur as targets for modular embeddings. They have been studied extensively in \cite{MR0204426} in rather classical language; however, for our approach the adelic formalism developed by Deligne \cite{MR0498581, MR546620} is more suitable. We state all definitions necessary for our setup, but for their incorporation into a general theory the reader is advised to consult \cite{MR0498581, MR546620, MR1044823, MR2192012}.

\subsection{The Shimura varieties themselves}\label{Subs:Shimura}

Let $F$ be a totally real number field and let $\mathscr{S}=\mathscr{S}_f\cup\mathscr{S}_{\infty }$ be its set of places. Let $B$ be a quaternion algebra over $F$ and let $\mathscr{R}(B)=\mathscr{R}_f(B)\cup\mathscr{R}_{\infty }(B)$ be the set of places where $B$ is ramified. Throughout we assume Eichler's condition: the set $\mathscr{P}=\mathscr{S}_{\infty }\smallsetminus\mathscr{R}_{\infty }(B)$ is nonempty.  For every $\varrho\in\mathscr{P}$ we fix an isomorphism
\begin{equation}\label{AlgebraIsomorphismOverReals}
B\otimes_{F,\varrho }\bbr\cong\mathrm{M}_2(\bbr ).
\end{equation}

Let $H$ be the linear algebraic group over $F$ with $H(F)=B^1$ (the group of invertible elements in $B$ with reduced norm one); let $G=\Res_{F/\bbq }H$. Then (\ref{AlgebraIsomorphismOverReals}) induces, for every $\varrho\in\mathscr{P}$, an isomorphism
 $H_F\otimes_{F,\varrho }\bbr\cong\SL_{2,\bbr }$, and for every $\varrho\in\mathscr{R}_{\infty }(B)$ we fix an isomorphism $H_F\otimes_{F,\varrho }\bbr\cong\operatorname{SU}_{2,\bbr }$. These glue to an isomorphism
\begin{equation}\label{ProductDecompositionGR}
G_{\bbr }\cong\prod_{\varrho\in\mathscr{S}_{\infty }}H_F\otimes_{F,\varrho }\bbr \to\prod_{\varrho\in\mathscr{P}}\SL_{2,\bbr }\times\prod_{\varrho\in\mathscr{R}_{\infty }(B)}\operatorname{SU}_{2,\bbr }.
\end{equation}
Let $G^{\mathrm{ad}}$ be the adjoint group of $G$, and similarly for~$H$, so that $\gad (\bbq )=\had (F)=B^{\times }/F^{\times }$. Then the isomorphism (\ref{ProductDecompositionGR}) induces an isomorphism
\begin{equation}\label{ProductDecompositionGRad}
G^{\mathrm{ad}}_{\bbr }\cong\prod_{\varrho\in\mathscr{S}_{\infty }}H^{\mathrm{ad}}_F\otimes_{F,\sigma }\bbr \to\prod_{\varrho\in\mathscr{P}}\PGL_{2,\bbr }\times\prod_{\varrho\in\mathscr{R}_{\infty }(B)}\operatorname{PU}_{2,\bbr }.
\end{equation}
Let $\bbs =\Res_{\bbc /\bbr }\bbg_m$ be the Deligne torus (cf.~Section~\ref{Sect:HodgeStructures}), and let $h_1\colon\bbs\to\PGL_{2,\bbr }$ be the standard homomorphism, given on $\bbr$-points by
$$\bbc^\times\to\PGL_2(\bbr ),\qquad a+b\mathrm{i}\mapsto\begin{pmatrix}a& -b\\ b& a\end{pmatrix}\bmod\bbr^{\times }.$$
Let $h\colon\bbs\to G^{\mathrm{ad}}_{\bbr }$ be the homomorphism which in the product decomposition (\ref{ProductDecompositionGRad}) is given by $h_1$ on each $\PGL_{2,\bbr }$-factor and by the trivial homomorphism on each $\operatorname{PU}_{2,\bbr }$-factor. Let $X\subset\Hom (\bbs ,\gad_{\bbr })$ be the $\gad (\bbr )^+$-conjugacy class containing~$h$.

The pair $(G,X)$ is a connected Shimura datum in the following sense:
\begin{Definition}
A \emph{connected Shimura datum} is a pair $(G,X)$ consisting of a semisimple algebraic group $G$ over $\bbq$ and a $\gad (\bbr )^+$-conjugacy class $X$ of homomorphisms $h\colon\bbs\to \gad_{\bbr }$ such that the following conditions are satisfied for one (and hence all) $h\in X$:
\begin{enumerate}
\item No characters other than $1$, $z/\widebar{z}$ and $\widebar{z}/z$ occur in the representation of $\bbs$ on $\operatorname{Lie}\gad_{\bbc }$.
\item $\operatorname{ad}h(\mathrm{i})$ is a Cartan involution on~$G^{\ad }$.
\item $h$ projects nontrivially to every $\bbq$-factor of~$G^{\ad }$.
\end{enumerate}
\end{Definition}
The definition given here is the `alternative definition' \cite[Definition~4.22]{MR2192012}. It is also used in \cite[p.~93]{MR931206} on which we heavily rely.

\subsubsection*{Congruence subgroups}

Recall that an automorphism of a finite-dimensional $\bbc$-vector space is \emph{neat} if its eigenvalues generate a torsion-free subgroup of $\bbc^{\times }$, and an arithmetic subgroup $\varGamma\subset R(\bbq )$ of a linear algebraic group $R$ over $\bbq$ is neat if the image of $\varGamma$ under one (equivalently, any) faithful rational representation of $R$ has only neat elements. Every sufficiently small congruence subgroup of $R(\bbq )$ is neat by \cite[17.4]{MR0244260}. Note that neat subgroups are torsion-free.

Let $G$ be a semisimple linear algebraic group over~$\bbq$. The set of neat congruence subgroups of $G(\bbq )$ is denoted by $\tilde{\Sigma }(G)$. The completion of $G (\bbq )$ with respect to the topology induced by the subgroups in $\tilde{\Sigma }(G)$ is canonically identified with $G(\bba^f)$.

The set of neat arithmetic subgroups of $\gad (\bbq )^+$ that contain the image of a congruence subgroup of $G(\bbq )$ is denoted by $\Sigma (G)$. Note that these are not necessarily congruence subgroups of~$\gad (\bbq )^+$. If $\tilde{\varGamma }\in\tilde{\Sigma }(G)$ then its image in $\gad (\bbq )$ is an element of $\Sigma (G)$ denoted by $\operatorname{ad}\tilde{\varGamma }$. The completion of $\gad (\bbq )^+$ with respect to the topology induced by these subgroups is denoted by $\mathcal{A}(G)$. Note that this is the group denoted by $\gad (\bbq )^+\,\hat{\,}$ in \cite{MR546620} and~\cite{MR931206}. If $G$ is simply connected (as in all cases we consider) the canonical maps $\operatorname{ad}\colon G(\bba^f)\to\mathcal{A}(G)$ and $\gad (\bbq )^+\to\mathcal{A}(G)$ define an isomorphism of topological groups
\begin{equation}\label{GAasAmalgamatedProduct}
G(\bba^f)\ast_{G(\bbq )}\gad (\bbq )^+\cong\mathcal{A}(G),
\end{equation}
see \cite[2.1.6.2]{MR546620}.

\subsubsection*{Quotients}

The space $X$ admits a natural structure as a complex manifold, in fact as an Hermitian symmetric domain. The operation of $\gad (\bbr )^+$ on $X$ by conjugation identifies $\gad (\bbr )^+$ with the group of biholomorphisms $X\to X$. Each arithmetic subgroup $\varGamma\subset\gad (\bbq )^+$ containing a group in $\Sigma (G)$ operates freely and properly discontinuously on $X$; the quotient $\varGamma\backslash X$ is a normal complex space, and by \cite[Theorem~3.10]{MR0338456} it admits a unique structure as a normal quasiprojective complex variety. We denote this variety by $\Sh^0_{\varGamma }(G,X)$. If $\varGamma\in\Sigma (G)$ this variety is smooth, and $X$ is the universal cover of $\varGamma\backslash X=\Sh^0_{\varGamma }(G,X)(\bbc )$. The system consisting of the $\Sh_{\varGamma }^0(G,X)$ for varying $\varGamma\in\Sigma (G)$ together with the obvious `forgetful' maps $\Sh_{\varGamma '}^0(G,X)\to\Sh_{\varGamma }^0(G,X)$ for $\varGamma '\subseteq\varGamma$, as well as its limit
\begin{equation*}
\Sh^0(G,X)=\varprojlim_{\varGamma\in\Sigma (G)}\Sh_{\varGamma }^0(G,X)=\varprojlim_{\tilde{\varGamma }\in\tilde{\Sigma }(G)}\Sh_{\operatorname{ad}\tilde{\varGamma }}^0(G,X)
\end{equation*}
(limits in the category of $\bbc$-schemes), is called the \emph{connected Shimura variety} associated with $(G,X)$. There is a natural left action of $\gad (\bbq )^+$ on $\Sh^0(G,X)$, given on finite levels by
\begin{equation*}
g\colon\varGamma\backslash X\to g\varGamma g^{-1}\backslash X,\qquad [x]\mapsto [gx].
\end{equation*}
By continuity it extends to an action of $\mathcal{A}(G)$ on $\Sh^0(G,X)$. If $\varGamma\in\Sigma (G)$ then the finite type variety $\Sh^0_{\varGamma }(G,X)$ can be retrieved as the quotient $\hat{\varGamma }\backslash\Sh^0(G,X)$ where $\hat{\varGamma }$ is the topological closure of $\varGamma$ in $\mathcal{A}(G)$.

\subsubsection*{Adelic description}

Let $\tilde{\varGamma }\in\tilde{\Sigma }(G)$ and let $K$ be its closure in $G(\bba^f)$. Letting $G(\bbq )$ operate from the left and $K$ from the right on $X\times G(\bba^f)$ by
$$q(x,g)k=(\operatorname{ad}(q)x,qgk)$$
we can form the double quotient $G(\bbq )\backslash X\times G(\bba^f)/K$. From the general formalism of Shimura varieties it follows that the map
\begin{equation*}
\operatorname{ad}\tilde{\varGamma }\backslash X\to G(\bbq )\backslash X\times G(\bba^f)/K,\qquad [x]\mapsto [x,1]
\end{equation*}
is biholomorphic. Passing to the limit over smaller and smaller $\tilde{\varGamma }$ we obtain a bijection
\begin{equation}\label{ShZeroAsSimpleQuotient}
\Sh^0(G,X)(\bbc )\cong G(\bbq )\backslash X\times G(\bba^f)
\end{equation}
that can even be considered as a biholomorphism for suitable topologies and complex structures on both sides, see \cite[Section~2.7]{MR546620}. The action of $\mathcal{A}(G)$ on $\Sh^0(G,X)$ translates via (\ref{GAasAmalgamatedProduct}) and (\ref{ShZeroAsSimpleQuotient}) to the following two left actions on $G(\bbq )\backslash X\times G(\bba^f)$:
\begin{enumerate}
\item $G(\bba^f)$ operates by $a\colon [x,g]\mapsto [x,ga^{-1}]$;
\item $\gad (\bbq )^+$ operates by $q\colon [x,g]\mapsto [qx,\operatorname{ad}(q)(g)]$.
\end{enumerate}
Since these two actions agree on $G(\bbq )$, they indeed glue to an action by $\mathcal{A}(G)$.

\subsubsection*{Dichotomy: projective or modular}

Let $\varGamma\in\Sigma (G)$ and let $S=\Sh^0_{\varGamma }(G,X)$. Then $S$ is projective if and only if $B$ is a division algebra; this is easily seen using the compactness criterion \cite[Theorem~11.6]{MR0147566} of Borel and Harish-Chandra. Otherwise $B$ has to be isomorphic to $\mathrm{M}_2(F)$, hence $G$ is isomorphic to $\Res_{F/\bbq }\SL_2$. The associated Shimura varieties are then called \emph{Hilbert--Blumenthal} varieties and can be interpreted as moduli spaces for abelian varieties $A$ of dimension $[F:\bbq ]$ with a ring homomorphism $F\to\End (A)\otimes\bbq$, together with certain level structures whose exact form depends on~$\varGamma$. For more about these varieties see~\cite{MR930101}.

\subsubsection*{The compact dual Hermitian symmetric space}

For each $h\colon\bbs\to \gad_{\bbr }$ in $X$ we can consider the associated Hodge filtration homomorphism $\mu_h\colon\bbg_{m,\bbc}\to \gad_{\bbc }$. The $\gad (\bbc )$-conjugacy class $\check{X}\subset\Hom (\bbg_{m,\bbc },\gad_{\bbc })$ containing $\mu_h$ depends only on $X$ and not on the chosen representative $h\in X$. The space $\check{X}$ has a natural structure as a compact complex manifold, and in fact as a complex algebraic variety. The \emph{Borel embedding}
\begin{equation*}
\beta\colon X\hookrightarrow\check{X},\qquad h\mapsto \mu_h,
\end{equation*}
realises $X$ as an open complex submanifold of~$\check{X}$.

\subsubsection*{Classical description}

We let $\gad (\bbr )^+$ act on $\bbh^{\mathscr{P}}$ by coordinate-wise M\"{o}bius transformations via the projection
\begin{equation*}
\gad (\bbr )^+\overset{(\ref{ProductDecompositionGRad})}{\cong }\prod_{\varrho\in\mathscr{P}}\PGL_2(\bbr )^+\times\prod_{\varrho\in\mathscr{R}_{\infty }(B)}\operatorname{PU}_2(\bbr )\twoheadrightarrow\prod_{\varrho\in\mathscr{P}}\PGL_2(\bbr )^+.
\end{equation*}
Then for each $h\in X$ the image $h(\bbc^{\times })\subset \gad (\bbr )^+$ has a unique common fixed point in $\bbh^{\mathscr{P}}$, and sending $h$ to that fixed point defines a biholomorphism $X\to\bbh^{\mathscr{P}}$.

Similarly, the identification (\ref{ProductDecompositionGRad}) defines a projection
\begin{equation*}
\gad (\bbc )\overset{(\ref{ProductDecompositionGRad})}{\cong }\prod_{\varrho\in\mathscr{P}}\PGL_2(\bbc )\times\prod_{\varrho\in\mathscr{R}_{\infty }(B)}\operatorname{PU}_2(\bbc )\twoheadrightarrow\prod_{\varrho\in\mathscr{P}}\PGL_2(\bbc ).
\end{equation*}
We let $\gad (\bbc )$ operate via this projection and by coordinate-wise M\"{o}bius transformations on $\bbp^1(\bbc )^{\mathscr{P}}$. For each $\mu\colon\bbg_m\to \gad_{\bbc }$ in $\check{X}$ the image $\mu (\bbc )\subset \gad (\bbc )$ has a unique fixed point in $\bbp^1(\bbc )^{\mathscr{P}}$, and sending $\mu$ to that fixed point defines a biholomorphism $\check{X}\to\bbp^1(\bbc )^{\mathscr{P}}$.

The maps constructed in this paragraph fit in a commutative diagram
\begin{equation*}
\xymatrix{
X \ar[d]^-{\simeq } \ar@{^{(}->}[r]^-{\beta } & \check{X} \ar[d]_-{\simeq } \\
\bbh^{\mathscr{P}} \ar@{^{(}->}[r] &\bbp^1(\bbc )^{\mathscr{P}}
}
\end{equation*}
with the natural inclusion on the lower horizontal line; the morphisms in this diagram are equivariant for the corresponding morphisms in the commutative diagram
\begin{equation*}
\xymatrix{
\gad (\bbr )^+\ar[d]^-{\simeq } \ar@{^{(}->}[r] & \gad (\bbc ) \ar[d]_-{\simeq } \\
(\PGL_2(\bbr )^+)^{\mathscr{P}} \ar@{^{(}->}[r] &\PGL_2(\bbc )^{\mathscr{P}}.
}
\end{equation*}

\subsection{Automorphic line bundles}

Let $\mathfrak{L}\to\bbp^1$ be the total space of the line bundle $\mathcal{O}(1)$ on $\bbp^1$. Recall that the fibre $\mathfrak{L}_x$ over a point $x\in\bbp^1(\bbc )$ represented by a one-dimensional subspace $W\subset\bbc^2$ is the dual space $W^{\ast }=\Hom (W,\bbc )$. Therefore the points of $\mathfrak{L}(\bbc )$ are pairs $(W,\lambda )$ where $W\subset\bbc^2$ is a line and $\lambda\colon W\to\bbc$ is a linear form. Then
$$\gamma \cdot (W,\lambda )\defined (\gamma (W),\lambda\circ\gamma^{-1})$$
defines an action of $\SL_2(\bbc )$ on $\mathfrak{L}(\bbc )$ which is fact an algebraic action of $\SL_2$ on $\mathfrak{L}$ defined over $\bbq$ and compatible with the vector bundle structure on~$\mathfrak{L}$.

For each embedding $\varrho\in\mathscr{P}$ let $\mathfrak{L}_{\varrho }$ be the line bundle on $\check{X}$ which is the pullback of $\mathfrak{L}$ along the composition
\begin{equation*}
\check{X}\cong\bbp^1(\bbc )^{\mathscr{P}}\overset{\operatorname{pr}_{\varrho }}{\twoheadrightarrow }\bbp^1(\bbc ).
\end{equation*}
Since this holomorphic map is equivariant for the corresponding projection on groups,
\begin{equation*}
G (\bbc )\twoheadrightarrow\SL_2(\bbc )^{\mathscr{P}}\overset{\operatorname{pr}_{\varrho }}{\twoheadrightarrow}\SL_2(\bbc ),
\end{equation*}
the line bundle $\mathfrak{L}_{\varrho }\to\check{X}$ becomes a $G_{\bbc }$-vector bundle in a natural way. That is, $G_{\bbc }$ operates on $\mathfrak{L}_{\varrho }$ as an algebraic variety in a way which is compatible with the structure map $\mathfrak{L}_{\varrho }\to\check{X}$ and which respects the vector space structures on the fibres.

This action restricts to a holomorphic action of $G (\bbr )$ on $\mathfrak{L}_{\varrho }|_X\to X$. For every subgroup $\tilde{\varGamma }\in\tilde{\Sigma}(G)$ this vector bundle descends to a holomorphic vector bundle
$$L_{\varrho ,\tilde{\varGamma }}\defined\tilde{\varGamma }\backslash (\mathfrak{L}_{\varrho }|_X)\to\operatorname{ad}\tilde{\varGamma }\backslash X=\Sh^0_{\operatorname{ad}\tilde{\varGamma }}(G,X)(\bbc ).$$
\begin{Lemma}\label{Lemma:AlgebraicStructureUnique}
Assume that $\lvert\mathscr{P}\rvert =\dim S>1$. Then there exists a unique algebraic structure $\mathcal{L}_{\varrho ,\tilde{\varGamma }}$ on the bundle $L_{\varrho ,\tilde{\varGamma }}$ compatible with the given holomorphic structure. Furthermore, any holomorphic global section of $(\mathcal{L}_{\varrho ,\tilde{\varGamma }})^{\otimes n}\otimes\Omega^p_S$ is algebraic.
\end{Lemma}
\begin{proof}
This is contained in \cite[III. Lemma~2.2]{MR1044823}.
\end{proof}
Also in the one-dimensional case there exists a canonical (but no more unique) algebraic line bundle underlying $L_{\varrho ,\tilde{\varGamma }}$ (see \cite[Section~4]{MR931206} for details); we denote its sheaf of sections by $\mathcal{L}_{\varrho ,\tilde{\varGamma }}$ or, if $\tilde{\varGamma }$ is clear from the context, simply by $\mathcal{L}_{\varrho }$. Furthermore, the line bundles $\mathcal{L}_{\varrho ,\tilde{\varGamma }}$ for varying $\tilde{\varGamma }$ are compatible in the obvious way, hence we obtain a line bundle $\mathcal{L}_{\varrho }$ on $\Sh^0(G,X)$ together with a continuous action of $G(\bba^f)$ lifting that on $\Sh^0(G,X)$. Note this does not extend to an action of $\mathcal{A}(G)$ on $\mathcal{L}_{\varrho }$.

\begin{Remark}\label{LRhoAndModularForms}
The bundle $\mathcal{L}_{\varrho ,\tilde{\varGamma }}$ is the holomorphic line bundle over $\operatorname{ad}\tilde{\varGamma }\backslash X$ with automorphy factors
$$j_{\varrho }(\tilde{\gamma} ,z) =c_{\varrho }z_{\varrho }+d_{\varrho }$$
where $z=(z_{\varrho })_{\varrho }\in\bbh^{\mathscr{P}}$ and
\begin{equation}\label{GammaMapsToMatrix}
\tilde{\gamma }\mapsto\begin{pmatrix}a_{\varrho }& b_{\varrho }\\ c_{\varrho }&d_{\varrho }\end{pmatrix}
\end{equation}
under the $\varrho$-th factor projection in~(\ref{ProductDecompositionGR}). That is, holomorphic sections of $\mathcal{L}_{\varrho }$ over an open subset $U\subseteq S(\bbc )$ are the same as holomorphic functions $f\colon V\to\bbc$, where $V$ is the preimage of $U$ in $X$, such that
$$f(\tilde{\gamma} z)=j_{\varrho }(\tilde{\gamma} ,z)f(z)$$
for all $\tilde{\gamma }\in\tilde{\varGamma }$ and $z\in V$. Hence global sections of $\mathcal{L}_{\varrho ,\tilde{\varGamma  }}$ can be identified with holomorphic modular forms for $\tilde{\varGamma }$ of weight one in direction $\varrho$ and weight zero in all other directions.
\end{Remark}
\begin{Remark}
The reason why we work with connected Shimura varieties instead of the more customary nonconnected Shimura varieties defined by the groups $G_0$ with $G_0(\bbq )=B^{\times }$ is that the theory of automorphic vector bundles is much simpler in the connected case, with a necessity to use stacks in the nonconnected case, cf.~\cite[Section~III.8]{MR1044823}. In our situation, any congruence subgroup of $G_0(\bbq )=B^{\times }$ contains a finite index subgroup of $\mathfrak{o}_K^{\times }$ and hence elements that operate trivially on $X$ but not on~$\mathfrak{L}_{\varrho }$.
\end{Remark}

Now fix some subgroup $\tilde{\varGamma }\in\tilde{\Sigma }(G)$ with image $\varGamma =\operatorname{ad}\tilde{\varGamma }\in\Sigma (G)$ and set $S=\Sh^0_{\varGamma }(G,X)$. The line bundles $\mathcal{L}_{\varrho }=\mathcal{L}_{\varrho ,\tilde{\varGamma }}$ on $S$ are related to the cotangent bundle~$\Omega^1_S$. Recall that the cotangent bundle of $\bbp^1$ is \emph{canonically} isomorphic to $\mathfrak{L}^{\otimes 2}$. Hence the cotangent bundle of $X\simeq\bbp^1(\bbc )^{\mathscr{P}}$ is canonically isomorphic to
$$\bigoplus_{\varrho \in\mathscr{P}}\mathfrak{L}_{\varrho }^{\otimes 2}$$
as a $G_{\bbc }$-vector bundle. This in turn implies that there is a canonical isomorphism of algebraic vector bundles on~$S$:
\begin{equation*}
\alpha\colon\bigoplus_{\varrho \in\mathscr{P}}\mathcal{L}_{\varrho }^{\otimes 2}\cong\Omega^1_S.
\end{equation*}
From this we can easily construct a Higgs field $\vartheta_{\varrho }\colon \mathcal{L}_{\varrho }\to \mathcal{L}^{-1}_{\varrho }\otimes \Omega_S^1$ as the composition
\begin{equation*}
\vartheta_{\varrho }\colon \mathcal{L}_{\varrho }\cong \mathcal{L}^{-1}_{\varrho }\otimes \mathcal{L}^{\otimes 2}_{\varrho }\hookrightarrow \mathcal{L}^{-1}_{\varrho }\otimes\bigoplus_{\sigma\in\mathscr{P} }\mathcal{L}_{\!\sigma }^{\otimes 2}\overset{\mathrm{id}\otimes\alpha }{\longrightarrow }\mathcal{L}^{-1}_{\varrho }\otimes\Omega_S^1.
\end{equation*}
These Higgs fields come from variations of Hodge structure:

For every $\varrho\in\mathscr{P}$ let $\bbw_{\varrho }$ be the pullback of the universal variation $\bbw$ on $\bbh$ along the projection $\operatorname{pr}_{\varrho }\colon\bbh^{\mathscr{P}}\to\bbh$. I.e., it is the polarised variation of real Hodge structure on $\bbh^{\mathscr{P}}$ whose underlying local system is the trivial local system $W=\bbr^2$ with the symplectic form $\psi$ as in~(\ref{StandardPolarisation}), and whose Hodge decomposition is of the form
$$(\bbw_{\varrho })_{z,\bbc }=W_z^{1,0}\oplus W_z^{0,1},\qquad W_z^{1,0}=\bbc\cdot {z_{\varrho }\choose 1}$$
(here $z=(\ldots ,z_{\varrho },\ldots )\in\bbh^{\mathscr{P}}$).

In complete analogy to the one-dimensional case, the operation of $\SL_2(\bbr )^{\mathscr{P}}$ by co\-ord\-ina\-te-wise M\"{o}bius transformations extends to an action on~$\bbw_{\varrho }$. For this reason, $\bbw_{\varrho }$ can be interpreted as a real variation of elliptic curve type on $X$ with an action of $G (\bbr )$. It therefore descends to a variation $\bbw_{\varrho ,\tilde{\varGamma }}$ on~$S$ whose monodromy representation is given by~(\ref{GammaMapsToMatrix}).

\begin{Lemma}\label{AutomorphicBundleIsHiggsBundleForVHS}
Suppress all subscripts $\tilde{\varGamma }$ for better readability. There is a canonical isomorphism of holomorphic line bundles $\mathrm{F}^1\mathcal{W}_{\varrho }\to\mathcal{L}_{\varrho  }^{\an }$ such that the diagram
\begin{equation*}
\xymatrix{
\mathrm{F}^1\mathcal{W}_{\varrho  } \ar[d]^-{\simeq}\ar[r]^-{\eta } & (\mathrm{F}^1\mathcal{W}_{\varrho  })^{-1}\otimes\Omega_S^1 \ar[d]^-{\simeq }\\
\mathcal{L}_{\varrho  }^{\an } \ar[r]_-{\vartheta^{\an }} & \mathcal{L}_{\varrho  }^{-1,\,\an }\otimes\Omega_S^1
}
\end{equation*} 
commutes up to a sign. Here $\eta$ is the Higgs field associated with the variation of Hodge structure~$\bbw_{\varrho }$.
\end{Lemma}
\begin{proof}
These isomorphisms can be realised as follows: a lift of $\mathcal{L}^{\an }_{\varrho }\to\mathrm{F}^1\mathcal{W}_{\varrho }$ to $X$ sends an automorphic function $f$ as in Remark~\ref{LRhoAndModularForms} to the section ${z_{\varrho }f(z)\choose f(z)}$ of $\mathrm{F}^1\mathcal{W}_{\varrho }\subset\mathcal{W}_{\varrho }$, and a lift of $\mathcal{W}_{\varrho }/\mathrm{F}^1\mathcal{W}_{\varrho }\to\mathcal{L}^{-1,\,\an }_{\varrho }$ to $X$ sends the equivalence class of ${f(z)\choose g(z)}$ to the function $z_{\varrho }g(z)-f(z)$ which is automorphic for the system of automorphy factors $j_{\varrho }(\tilde{\gamma },z)^{-1}$.
\end{proof}

\subsubsection*{Hilbert--Blumenthal case}

These objects have a more geometric description in the case where $B=\mathrm{M}_2(F)$. Recall that then for any $\varGamma =\ad\tilde{\varGamma }$ with $\tilde{\varGamma }\in\tilde{\Sigma }(G)$ the variety $S=\Sh^0_{\varGamma }(G,X)$ has an interpretation as a fine moduli space for abelian varieties with certain PEL structures, in particular there is a universal family $p\colon A\to S$ of polarised abelian varieties of rank $[F:\bbq ]$ together with a ring homomorphism $F\hookrightarrow\End (A/S)\otimes\bbq$. The polarised variation $\mathrm{R}^1p_{\ast }\bbr$ decomposes as a direct sum
\begin{equation*}
\mathrm{R}^1p_{\ast }\bbr =\bigoplus_{\varrho\colon F\to\bbr }(\mathrm{R}^1p_{\ast }\bbr )_{\varrho }
\end{equation*}
where $a\in F$ acts on the summand with subscript $\varrho$ as multiplication by $\varrho (a)$. A direct calculation shows that
\begin{equation}\label{WRhoAsSummandInRpast}
\mathrm{R}^1p_{\ast }\bbr \cong\bbw_{\varrho }.
\end{equation}
The holomorphic vector bundle underlying $\mathrm{R}^1p_{\ast }\bbc$ is canonically isomorphic to the relative de~Rham cohomology
$$\mathcal{H}^1_{\mathrm{dR}}(A/S)=\mathrm{R}^1p_{\ast }(\Omega_{A/S}^{\bullet})$$
with the Hodge filtration given by the short exact sequence
\begin{equation}\label{deRhamSESTotal}
0\to p_{\ast }\Omega_{A/S}^1\to\mathcal{H}^1_{\mathrm{dR}}(A/S)\to\mathrm{R}^1p_{\ast }\mathcal{O}_A\to 0
\end{equation}
and the Gauss--Manin connection $\nabla^{\mathrm{GM}}\colon\mathcal{H}^1_{\mathrm{dR}}(A/S)\to\mathcal{H}^1_{\mathrm{dR}}(A/S)\otimes\Omega_{S/\bbc }^1$. All these structures respect the $F$-action on cohomology and hence decompose into eigenspaces.
\begin{Proposition}
Let $S$ be a Hilbert--Blumenthal variety with universal family $p\colon A\to S$. Let $\mathcal{H}_{\varrho }$ be the subbundle of $\mathcal{H}^1_{\mathrm{dR}}(A/S)$ where $F$ operates through~$\varrho$, similarly for $(\Omega_{A/S}^1)_{\varrho }$ and $(\mathrm{R}^1p_{\ast }\mathcal{O}_A)_{\varrho }$. Then (\ref{deRhamSESTotal}) restricts to a short exact sequence
\begin{equation*}
0\to p_{\ast }(\Omega_{A/S}^1)_{\varrho }\to\mathcal{H}_{\varrho }\to (\mathrm{R}^1p_{\ast }\mathcal{O}_A)_{\varrho }\to 0
\end{equation*}
and the Gauss--Manin connection restricts to a connection
$$\nabla_{\varrho }^{\mathrm{GM}}\colon\mathcal{H}_{\varrho }\to\mathcal{H}_{\varrho }\otimes\Omega^1_S.$$
Furthermore there exist isomorphisms of algebraic line bundles
\begin{equation*}
\mathcal{L}_{\varrho }\cong p_{\ast }(\Omega_{A/S}^1)_{\varrho }\qquad\text{and}\qquad\mathcal{L}_{\varrho }^{-1}\cong (\mathrm{R}^1p_{\ast }\mathcal{O}_A)_{\varrho }
\end{equation*}
sending the Higgs field $\vartheta_{\varrho }\colon\mathcal{L}_{\varrho }\to\mathcal{L}_{\varrho }^{-1}\otimes\Omega_S^1$ to the negative of the Higgs field associated with the connection~$\nabla^{\mathrm{GM}}_{\varrho }$.
\end{Proposition}
\begin{proof}
If $\dim S>1$ (equivalently, if $F\neq\bbq$) the existence of these isomorphisms only needs to be checked in the analytic category by Lemma~\ref{Lemma:AlgebraicStructureUnique}; this is accomplished by combining Lemma~\ref{AutomorphicBundleIsHiggsBundleForVHS} with~(\ref{WRhoAsSummandInRpast}). The remaining case $F=\bbq$ is easy.
\end{proof}

\section{Geometry of modular embeddings}

\noindent In this section we study regular maps from algebraic curves to quaternionic Shimura varieties as in the previous section.

\subsection{Modular embeddings}\label{Sect:ME}

Let $(G,X)$ be a quaternionic connected Shimura datum as before, associated with a quaternion algebra $B$ over a totally real number field~$F$. We shall consider triples $(\varDelta ,\varphi ,\tif )$ of the following form:
\begin{enumerate}
\item $\varDelta$ is a lattice in $\PGL_2(\bbr )^+$;
\item $\varphi$ is a group homomorphism $\varphi\colon\varDelta\to\gad (\bbq )^+$ whose image is contained in an arithmetic subgroup of $\gad (\bbq )^+$;
\item $\tif$ is a holomorphic map $\tif\colon\bbh\to X$ which is equivariant for $\varphi$, i.e.\ which satisfies
\begin{equation*}
\tif (\delta z)=\varphi (\delta )\tif (z)\qquad\text{for all }\delta\in\varDelta ,\, z\in\bbh.
\end{equation*}
\end{enumerate}
Using the classical description of $(G,X)$ we can interpret $\tif$ as a map $\bbh\to\bbh^{\mathscr{P}}$; we denote the coordinate function $\bbh\to\bbh$ for $\varrho\in\mathscr{P}$ by $\tif_{\varrho }$. Similarly we can view $\varphi$ as a homomorphism $\varDelta\to\gad (\bbr )^+$, and via (\ref{ProductDecompositionGRad}) we obtain for every $\varrho\in\mathscr{P}$ a homomorphism $\varphi_{\varrho }\colon\varDelta\to\PGL_2(\bbr )^+$.
\begin{PropDef}\label{Def:ME}
Let $(G,X)$ be a quaternionic Shimura datum associated with $B/F$, let $(\varDelta, \varphi ,\tif )$ as above, and let $\varrho\in\mathscr{P}$. The following are equivalent:
\begin{enumerate}
\item $\varphi_{\varrho }\colon\varDelta\to\PGL_2(\bbr )^+$ is conjugate by an element of $\PGL_2(\bbr )^+$ to the identity embedding $\varDelta\hookrightarrow\PGL_2(\bbr )^+$.
\item $\tif_{\varrho }\colon\bbh\to\bbh$ is a biholomorphism.
\end{enumerate}
If this is the case we call $(\varphi ,\tif )$ a \emph{modular embedding} for $\varDelta$ with respect to~$\varrho$.
\end{PropDef}
\begin{proof}
The equivalence of the two conditions is a simple consequence of the Schwarz--Pick Lemma.
\end{proof}
If $(\varphi ,\tif )$ is a modular embedding for $\varDelta$ we can adjust the isomorphism (\ref{ProductDecompositionGRad}), whose precise choice is not essential for any of our constructions, in such a way that $\varphi_{\varrho }\colon\varDelta\to\PGL_2(\bbr )^+$ is the identity embedding and $\tif_{\varrho }\colon\bbh\to\bbh$ is the identity. Whenever we consider a specific modular embedding we will therefore tacitly assume that this is the case.

Sometimes we will need a stricter condition:
\begin{equation}\label{MELiesInImageOfGQ}
\varphi (\varDelta )\text{ is contained in the image of }G(\bbq )\to\gad (\bbq )^+.
\end{equation}
Let $\varDelta\subset\PGL_2(\bbr )^+$ be a lattice and let $\tilde{\varDelta }$ be its preimage in $\SL_2(\bbr )$. The \emph{trace field} of $\varDelta$ is the subfield
$$\bbq (\tr\varDelta )=\bbq (\tr\tilde{\delta }\mid\tilde{\delta }\in\tilde{\varDelta})\subset\bbr$$
generated by all traces of elements of $\tilde{\varDelta }$, and the \emph{invariant trace field} of $\varDelta$ is the intersection of all trace fields of finite index subgroups of $\varDelta$. If $\varDelta^{(2)}$ is the subgroup generated by all $\delta^2$ for $\delta\in\varDelta$ (a normal subgroup of finite index), then the invariant trace field of $\varDelta$ equals $\bbq (\tr\varDelta^{(2)})$, see \cite[Theorem~3.3.4]{MR1937957}.
\begin{Proposition}\label{InclusionOfTraceFields}
Let $(G,X)$ as above be associated with $B/F$. Let $\varDelta\subset\PGL_2(\bbr )^+$ be a lattice such that there exists a modular embedding for $\varDelta$ with respect to $\varrho\in\mathscr{P}$.
\begin{enumerate}
\item The invariant trace field of $\varDelta$ is contained in $\varrho (F)$. All traces $\tr\tilde{\delta }$ for $\tilde{\delta} \in\tilde{\varDelta }$ are algebraic integers.
\item If (\ref{MELiesInImageOfGQ}) is true then also the trace field $\bbq (\tr\varDelta )$ is contained in $\varrho (F)$.
\end{enumerate}
\end{Proposition}
\begin{proof}
Let $K$ be the invariant trace field and let $\varDelta '\subseteq\varDelta $ be a subgroup of finite index with trace field~$K$. Let $A\subset\mathrm{M}_2(\bbr )$ be the $K$-subalgebra generated by $\tilde{\varDelta }'\subset\SL_2(\bbr )\subset\mathrm{M}_2(\bbr )$; this is a quaternion algebra over~$K$.

By assumption the image of $\varDelta$ in $\gad (\bbq )^+$ is contained in an arithmetic subgroup containing a group in $\Sigma (G)$. By further shrinking $\varDelta '$ we may assume that $\varphi (\varDelta ')\subseteq\varGamma =\ad\tilde{\varGamma }$ for some $\tilde{\varGamma }\in\tilde{\Sigma }(G)$.

For (i), assume that $\varphi_{\varrho }\colon\varDelta \to\PGL_2(\bbr )^+$ is the identity inclusion; then $\tilde{\varDelta }'$ is contained in the image of $\tilde{\varGamma }\subset B^1=G(\bbq )$ under the $\varrho$-th coordinate map in~(\ref{ProductDecompositionGR}). But then $K\subseteq\varrho (F)$ since the latter is the trace field of $\tilde{\varGamma }$, and the identity inclusion induces an isomorphism $A\otimes_{K,\varrho^{-1}}F\cong B$. The traces of elements in $\tilde{\varDelta }'$ are then of the form $\varrho (x)$ for algebraic integers $x\in\mathfrak{o}_F$, hence the eigenvalues of any element in $\tilde{\varDelta }'$ are algebraic units. Since every element in $\tilde{\varDelta }$ has a finite nontrivial power in $\tilde{\varDelta }'$ also the eigenvalues of elements of $\tilde{\varDelta }$ are algebraic units, hence the traces are algebraic integers.

A very similar argument proves~(ii).
\end{proof}

\begin{Examples}\label{ExamplesME}
\begin{enumerate}
\item If $\varDelta$ is an arithmetic group itself there is a tautological modular embedding. Recall that $\varDelta$ is arithmetic if and only if it is commensurable to a group of the form $\mathscr{O}_0^1$ for some order $\mathscr{O}_0$ in a quaternion algebra $A_0/K_0$ which splits at precisely one archimedean place. Then there is a connected Shimura datum $(G_0,X_0)$ with $G_0(\bbq )=A_0^1$ and a modular embedding $(\tif ,\varphi )$ with $\tif\colon\bbh\to X_0$ a biholomorphism and $\varphi\colon\varDelta\to\gad_0(\bbq )^+$ restricting to the canonical embedding on~$\varDelta^{(2)}$. Note that this can be extended to $\varDelta$ by the Skolem--Noether theorem which implies that $\varDelta$ is contained in the image of $(A_0^{\times })^+$ in $\PGL_2(\bbr )^+$, which is canonically isomorphic to $\gad_0 (\bbq )^+$.

We may, however, start with any totally real extension $K/K_0$ and set $A=A_0\otimes_{K_0}K$. From this we can easily construct a modular embedding for $\varDelta$ into $(G,X)$ with $G(\bbq )=A^1$ and $X\simeq\bbh^{[K:K_0]}$; each of its coordinates is (up to a change of coordinates) given by $\tif$ and $\varphi$ constructed before for $(G_0,X_0)$.
\item For a specific example, let $\varDelta$ be the \emph{Fricke group} for a rational prime $p$, which is the group generated by the images of
$$\left\{\begin{pmatrix} a& b\\ c&d\end{pmatrix}\in\SL_2(\bbz )\,\middle\lvert\, c\equiv 0\bmod p\right\}\quad\text{and}\quad\begin{pmatrix}0&-1\\ p &0\end{pmatrix}$$
in $\PGL_2(\bbq )^+$. Considered as a subgroup of $\PGL_2(\bbr )^+$, its invariant trace field is $K_0=\bbq$ but its trace field is $K=\bbq (\sqrt{p})$ because the lifts of the additional matrix generator to $\SL_2(\bbr )$ are
$$\pm\begin{pmatrix} 0&-\frac 1{\sqrt{p}}\\ \sqrt{p} &0\end{pmatrix}.$$
The associated quaternion algebra is $B=\mathrm{M}_2(K)$, hence we can choose $\varGamma$ to be a congruence subgroup of $\SL_2(K)$, i.e.\ a \emph{Hilbert--Blumenthal modular group} operating on $\bbh^2$. Now Galois conjugation operates trivially on the elements of $\varDelta$, hence the diagonal map $\bbh\to\bbh^2$ is a modular embedding for $\varDelta$, and the image in the Hilbert--Blumenthal surface $\varGamma\backslash\bbh^2$ is a \emph{Hirzebruch--Zagier cycle}, cf.~\cite{MR0453649}.

Note that the modular embedding into $\PGL_2(K)^+$ satisfies (\ref{MELiesInImageOfGQ}), but the tautological embedding into $\PGL_2(\bbq )^+$ does not.

\item In a similar way we can construct a modular embedding where $\varDelta$ is cocompact but the arithmetic subgroups of $\gad (\bbq )^+$ are not. Let $\varDelta$ be a cocompact arithmetic group derived from a quaternion algebra $A_0$ over $K_0$, and choose $K$ to be a totally real extension of $K_0$ that splits $A_0$. Then the construction in (ii) yields a modular embedding $(\tif ,\varphi )$ with $\varphi\colon\varDelta\to\gad (\bbq )^+=\PGL_2(K)^+$.
\item Let $p,q,r\in\bbn\cup\{\infty \}$ with
\begin{equation*}
\frac 1p+\frac 1q+\frac 1r<1,
\end{equation*}
where we set $\frac 1\infty =0$. Then there exists a lattice $\varDelta_{p,q,r}\subset\PGL_2(\bbr )^+$, unique up to conjugacy, with presentation
\begin{equation*}
\varDelta_{p,q,r} =\langle x,y,z\mid x^p=y^q=z^r=xyz=1\rangle
\end{equation*}
where `$x^{\infty }=1$' is interpreted as `$x$ is a nontrivial parabolic element of $\PGL_2(\bbr )^+$', and similarly for $y$ and~$z$. This group is called the \emph{triangle group of signature $(p,q,r)$}. Its trace field is
$$K_{p,q,r}=\bbq \bigg(\!\cos\frac\pi p,\cos\frac\pi q,\cos\frac\pi r\bigg)$$
and it generates a quaternion algebra $A_{p,q,r}$ over this trace field. Then there is a connected quaternionic Shimura datum $(G,X)$ with $G(\bbq )=A_{p,q,r}^1$ and a modular embedding consisting of $\tif\colon\bbh\to X$ and $\varphi\colon\varDelta\to\gad (\bbq )^+$. This modular embedding is constructed in three different ways in~\cite{MR1075639}, by Schwarz triangle functions, by hypergeometric differential equations and by period maps for special families of abelian varieties. It has the convenient property (\ref{MELiesInImageOfGQ}), as can be seen from the first of the three constructions in~\cite{MR1075639}.
\item The uniformising groups of \emph{Teichm\"{u}ller curves}, i.e.\ algebraic curves which are immersed in the moduli space $\mathcal{M}_g$ of smooth projective genus $g$ curves in a totally geodesic way for the Teichm\"{u}ller metric, admit modular embeddings. This was shown in \cite[Theorem~10.1]{MR1992827} for $g=2$ and in \cite[Corollary~2.11]{MR2188128} for the general case. These modular embeddings are studied in detail in~\cite{MoellerZagier2015}. The quaternion algebras in these modular embeddings are always of the form $\mathrm{M}_2(F)$ with $\varrho (F)$ being equal to both the trace field and the invariant trace field. The equality of the latter is \cite[Corollary~6.1.11]{Diplom}. This implies that they, too, have the property (\ref{MELiesInImageOfGQ}) that $\varphi (\varDelta )$ is contained in the image of $G(\bbq )=\SL_2(F)$.
\end{enumerate}
\end{Examples}
A modular embedding $\tif\colon\bbh\to X$ descends to a holomorphic map between the quotients $f\colon\varDelta\backslash\bbh\to\varGamma\backslash X$. By a theorem of Borel \cite[Theorem~3.10]{MR0338456} this is induced by a regular map of normal algebraic varieties $f\colon C\to S=\Sh_{\varGamma }^0(G,X)$. If $f$ arises in this fashion we say it is \emph{covered by a modular embedding}.

\subsection{Maps from curves to Shimura varieties}

Let $(G,X)$ be a connected Shimura datum defined by a quaternion algebra $B/F$ as before, let $\varGamma\in\Sigma (G)$ and let $S=\Sh_{\varGamma }^0(G,X)$. Let further $C$ be a smooth complex algebraic curve and let $f\colon C\to S$ be a morphism of algebraic varieties. From Liouville's Theorem we easily see that if the universal covering of $C^{\an }$ is $\bbc$ or $\bbp^1(\bbc )$ then $f$ is constant, so we assume from now on that $C$ is hyperbolic.

After choosing appropriate basepoints we can lift $f$ to a holomorphic map between the universal covers
\begin{equation*}
\tif\colon\bbh\to X
\end{equation*}
which is equivariant for the induced map on fundamental groups
\begin{equation*}
\tilde{\varphi }=f_{\ast }\colon\varDelta =\pi_1(C^{\an })\to\varGamma\subset G (\bbr ).
\end{equation*}
Using the classical description $X\simeq\bbh^{\mathscr{P}}$ and $G (\bbr )\twoheadrightarrow\SL_2(\bbr )^{\mathscr{P}}$ we can split these up into coordinate components
\begin{equation*}
\tif_{\varrho }\colon\bbh\to\bbh\qquad\text{and}\qquad\tilde{\varphi }_{\varrho }\colon\pi_1(C^{\an })\to\SL_2(\bbr ).
\end{equation*}
We can also pull back the various bundles on $S$ along~$f$. Hence we obtain on $C$, for every $\varrho\in\mathscr{P}$, a real variation $f^{\ast }\bbw_{\varrho }$ of elliptic curve type, a line bundle $f^{\ast }\mathcal{L}_{\varrho }$ and a Higgs field
\begin{equation}\label{HiggsPullback}
f^{\ast }\vartheta_{\varrho }\colon f^{\ast }\mathcal{L}_{\varrho }\to f^{\ast }\mathcal{L}^{-1}_{\varrho }\otimes f^{\ast }\Omega_S^1\overset{(\ast )}{\to }f^{\ast }\mathcal{L}^{-1}_{\varrho }\otimes\omega_C,
\end{equation}
where $f^{\ast }\Omega_S^1\to\Omega_C^1=\omega_C$ in $(\ast )$ is the canonical map. Note that since $f$ is an algebraic map and the Higgs field $\vartheta_{\varrho }$ on $S$ is algebraic, so is~(\ref{HiggsPullback}).

We first assume that $C$ is a smooth projective curve.
\begin{Theorem}\label{CriteriaModularEmbedding}
Let $C$ be a smooth projective curve of genus at least two, let $S=\Sh^0_{\ad\tilde{\varGamma }}(G,X)$ with $\tilde{\varGamma }\in\tilde{\Sigma }(G)$, and let $f\colon C\to S$ be a morphism of algebraic varieties. Fix some embedding $\varrho\in\mathscr{P}$. With the notations of this section, the following are equivalent.
\begin{enumerate}
\item $\tif$ is a modular embedding with respect to~$\varrho$.
\item The homomorphism $\varphi_{\varrho }\colon\pi_1(C^{\an })\to\SL_2(\bbr )$ is conjugate to a theta characteristic of $\varDelta =\pi_1(C^{\an })$.
\item The component map $\tif_{\varrho }\colon\bbh\to\bbh$ is a biholomorphism.
\item The Higgs field $f^{\ast }\vartheta_{\varrho }$ is maximal.
\item $f^{\ast }\mathcal{L}_{\varrho }$ is a theta characteristic on~$C$.
\item The degree of the line bundle $f^{\ast }\mathcal{L}_{\varrho }$ is equal to $-\frac 12\chi (C)$.
\end{enumerate}
If these conditions are satisfied, then $\varphi_{\varrho }\in\Theta^{\mathrm{uni}}(\varDelta )$ in (ii) and $f^{\ast }\mathcal{L}_{\varrho }\in\Theta (C)$ in (v) correspond to each other under the bijection $\Theta^{\mathrm{uni}}(\varDelta )\leftrightarrow\Theta (C)$ from section~\ref{Subs:Theta}.
\end{Theorem}
\begin{proof}
This follows directly from Proposition~\ref{CharacterisationMaximalHiggs}.
\end{proof}
Finally we note that the line bundle $\mathcal{M}_{\varrho }=\mathcal{L}_{\varrho }^{\otimes 2}$ is geometrically much easier to describe than $\mathcal{L}_{\varrho }$: it is a summand in the decomposition
$$\Omega_S^1=\bigoplus_{\varrho\in\mathscr{P}}\mathcal{M}_{\varrho }$$
coming from the product decomposition $X\cong\bbh^{\mathscr{P}}$. Furthermore, it depends only on the group $\varGamma\in\Sigma (G)$ instead of a choice of $\tilde{\varGamma }\in\tilde{\Sigma }(G)$.
\begin{Corollary}\label{CriteriaMEWithM}
Let $C$ be a smooth projective curve of genus at least two, let $S=\Sh^0_{\varGamma }(G,X)$ with $\varGamma \in\Sigma (G)$ and let $f\colon C\to S$ be a morphism of algebraic varieties. Fix some embedding $\varrho\in\mathscr{P}$. Then $f$ is covered by a modular embedding with respect to $\varrho$ if and only if $f^{\ast }\!\mathcal{M}_{\varrho }$ is isomorphic to~$\omega_C$ (as an algebraic, holomorphic or topological line bundle; these statements are all equivalent under these conditions).
\end{Corollary}
\begin{proof}
Assume first that $f^{\ast }\!\mathcal{M}_{\varrho ,\varGamma }\simeq\omega_C$, and let $\varGamma '\subseteq\varGamma$ be a subgroup with $\varGamma '=\ad\tilde{\varGamma }'$ for some $\tilde{\varGamma }'\in\tilde{\Sigma }(G)$. Let $C'\to C$ be the finite unramified covering corresponding to the subgroup $\varphi^{-1}(\varGamma ')\subseteq\varDelta$, and let $f'\colon C'\to S'=\Sh_{\varGamma '}^0(G,X)$ be the obvious lift of~$f$. Then by pullback $f'^{\ast }\!\mathcal{M}_{\varrho,\varGamma '}\simeq\omega_{C'}$, hence $f'^{\ast }\mathcal{L}_{\varrho ,\tilde{\varGamma }'}$ is a theta characteristic on~$C'$. By Theorem~\ref{CriteriaModularEmbedding} $f'$ is covered by a modular embedding, hence so is~$f$.

Now assume that $f$ is covered by a modular embedding, and choose $\varGamma '$ as before, but in addition normal in~$\varGamma$. Then we obtain an isomorphism of line bundles $f'^{\ast }\!\mathcal{M}_{\varrho ,\varGamma }\to\omega_{C'}$ which is by construction invariant under the finite group $\varGamma /\varGamma '$ acting on $C'$, hence it descends to an isomorphism $f^{\ast }\!\mathcal{M}_{\varrho ,\varGamma }\to\omega_C$.
\end{proof}

We now turn to the affine case.
\begin{Proposition}\label{Prop:AffineImpliesHB}
If $f\colon C\to S$ is covered by a modular embedding, where $C$ is a smooth affine complex curve, then $B\simeq\mathrm{M}_2(F)$, hence $S$ is a Hilbert--Blumenthal variety.
\end{Proposition}
Note that the reverse implication does not hold: there exist projective Shimura curves modularly embedded in non-projective Hilbert--Blumenthal varieties, cf.\ Example~\ref{ExamplesME}.(iii).
\begin{proof}
The preimage of the fundamental group of $C^{\an }$ in $\SL_2(\bbr )$ contains nontrivial unipotent elements. Hence so does $G(\bbq )$; but if $1\neq\gamma\in G(\bbq )=B^1$ is unipotent, then $\gamma -1\in B\smallsetminus\{ 0\}$ is nilpotent, hence $B$ cannot be a division algebra.
\end{proof}
We assume for the rest of this subsection that $B=\mathrm{M}_2(F)$. Recall that then $\bbw_{\varrho }$ can be described as a direct summand, defined by an eigenvalue equation, of the variation $\mathrm{R}^1p_{\ast }\bbr$, where $p\colon A\to S$ is the universal family, and the associated vector bundle with its Hodge filtration and its connection can be constructed in a purely algebraic way from the universal family. 
\begin{Lemma}
After possibly replacing $\varGamma$ by a suitable subgroup in $\Sigma (G)$ and correspondingly $C$ by a finite unramified covering we may assume that $f^{\ast }\bbw_{\varrho }$ has unipotent monodromy around the cusps of~$C$.
\end{Lemma}
\begin{proof}
Let $p_C\colon A_C\to C$ be the pullback of the universal family $p\colon A\to S$ along~$f$. Then by (\ref{WRhoAsSummandInRpast}) the variation $f^{\ast }\bbw_{\varrho }$ is a direct summand of the variation $\mathrm{R}^1p_{C,\ast }\bbr$ on~$C$, and such variations have quasi-unipotent monodromy around the cusps by \cite[III.\S 2, Th\'eor\`eme~2.3]{MR0417174}.
\end{proof}
\begin{Theorem}\label{GaussManinCriterion}
Let $S$ and $C=\Cbar\smallsetminus D$ as before, and let $f\colon C\to S$ be a regular map corresponding to an algebraic family of abelian varieties $p_C\colon A_C\to C$ with an inclusion $F\hookrightarrow\End (A_C/C)\otimes\bbq$. Denote $\varrho$-isotypical components of $F$-modules by subscripts~$\varrho$. Then, after possibly passing to finite unramified coverings, the Gauss--Manin connection $\nabla^{\mathrm{GM}}_{\varrho }$ on the vector bundle $(\mathcal{H}^1_{\mathrm{dR}}(A_C/C))_{\varrho }$ on $C$ has regular singularities with unipotent monodromy around the cusps of~$C$. Let $\eta_{\mathrm{ext}}\colon\mathcal{E}_{\mathrm{ext}}\to\mathcal{E}_{\mathrm{ext}}^{-1}\otimes\Omega_{\Cbar }(\log D)$ be the logarithmic Higgs bundle corresponding to the Deligne extension of $(\mathcal{H}^1_{\mathrm{dR}}(A_C/C))_{\varrho }$.

Then the following are equivalent:
\begin{enumerate}
\item $f$ is covered by a modular embedding for $\varrho$.
\item $\eta_{\mathrm{ext}}$ is an isomorphism.
\item $\mathcal{E}_{\mathrm{ext}}$ is a logarithmic theta characteristic on~$C$.
\item The degree of $\mathcal{E}_{\mathrm{ext}}$ is $-\frac 12\chi (C)$.
\end{enumerate}
\end{Theorem}
\begin{proof}
Since the logarithmic Higgs field $\eta_{\mathrm{ext}}$ is isomorphic to that associated with the variation $f^{\ast }\bbw_{\varrho }$, conditions (ii) and (iii) can be reformulated as statements about that variation. The $\varrho$-th component $\tif_{\varrho }\colon\bbh\to\bbh$ of the universal covering map $\tif$ is the period map for this variation. The theorem follows from Proposition~\ref{CharacterisationMaximalHiggs}.
\end{proof}

\begin{Remark}\label{Rem:AOneVsATwo}
It is desirable to compare Theorems \ref{CriteriaModularEmbedding} and~\ref{GaussManinCriterion}. Note that in the setup of Theorem~\ref{GaussManinCriterion} the line bundle $\mathcal{E}=\mathcal{E}_{\mathrm{ext}}|_C=p_{\ast }(\Omega_{A_C/C}^1)_{\varrho }$ is isomorphic as an algebraic line bundle to $f^{\ast }\mathcal{L}_{\varrho }$, with the Higgs fields $\eta =\eta_{\mathrm{ext}}|_C$ and $f^{\ast }\vartheta_{\varrho }$ corresponding to each other. Once we know that $f$ is covered by a modular embedding, there is a unique extension of $(f^{\ast }\mathcal{L}_{\varrho },f^{\ast }\vartheta_{\varrho })$ to a logarithmic Higgs bundle on $\Cbar$ with maximal Higgs field, but for an arbitrary regular map $f$ the situation is more complicated and there is no obvious way to choose an extension of $(f^{\ast }\mathcal{L}_{\varrho },f^{\ast }\vartheta_{\varrho })$ to a logarithmic Higgs bundle on~$\Cbar$. It seems that we need the subtler information encoded in the relative de Rham cohomology to decide how to extend the Higgs bundle to the cusps. Alternatively one might study a suitable extension $\widebar{f}\colon\Cbar\to\widebar{S}$ for some compactification $\widebar{S}$ of $S$, corresponding to a family of semi-abelian varieties over $\Cbar$, and try to obtain $\eta_{\mathrm{ext}}$ as a Kodaira--Spencer map for~$\widebar{f}$. This, however, seems technically much more difficult and is not necessary for our purposes.
\end{Remark}

\subsection{Adelic theory}

Let $\varDelta\subset\PGL_2(\bbr )^+$ be a semi-arithmetic lattice, let $\varGamma\subset\gad(\bbq )^+$ be an arithmetic subgroup containing the image of a congruence subgroup of $G(\bbq )$, and let $\varphi\colon\varDelta\to\varGamma$ and $\tif\colon\bbh\to X$ constitute a modular embedding with respect to $\varrho\colon F\to K\subset\bbr$.
\begin{Definition}
Let $\varDelta '\subseteq\varDelta$ be a subgroup. Then $\varDelta '$ is a \emph{congruence subgroup} of $\varDelta$ with respect to $\varphi$ if it contains a subgroup of the form $\varphi^{-1}(\varGamma ')$ with $\varGamma '\in\Sigma (G)$.
\end{Definition}
Recall that $\mathscr{O}=\mathfrak{o}_K\langle\varDelta\rangle$ is an order in a quaternion algebra~$A/K$. For every ideal $\mathfrak{n}$ of $K$ we define the \emph{principal congruence subgroup} $\varDelta (\mathfrak{n})$ as
$$\varDelta (\mathfrak{n})=\{\delta\in\varDelta\mid \delta -1\in\mathfrak{n}\mathscr{O}\} .$$
This is clearly a congruence subgroup.
\begin{Proposition}
Assume that $\varphi$ satisfies~(\ref{MELiesInImageOfGQ}). Then a subgroup of $\varDelta$ is a congruence subgroup with respect to $\varphi$ if and only if it contains a subgroup of the form $\varDelta (\mathfrak{n})$.
\end{Proposition}
\begin{proof}
Easy.
\end{proof}
Therefore for modular embeddings satisfying (\ref{MELiesInImageOfGQ}) the notion of congruence subgroup is intrinsic to $\varDelta$. In general it can be slightly more complicated. In any case the congruence subgroups define the neighbourhood basis of the identity for a topological group structure on $\varDelta$, and the completion $\hat{\varDelta }$, called the \emph{congruence completion}, is canonically isomorphic to
$$\hat{\varDelta }\cong\varprojlim_{\varDelta '}\varDelta /\varDelta ',$$
$\varDelta '$ running over all normal congruence subgroups of~$\varDelta$. This is a profinite group. Similarly, the congruence completion of $\varGamma$ is a profinite group $\hat{\varGamma }$ which is canonically isomorphic to the topological closure of $\varGamma$ in $\mathcal{A}(G)$, an open compact subgroup of the latter.
The homomorphism $\varphi\colon\varDelta\to\varGamma$ has a unique continuous extension $\hat{\varphi }\colon\hat{\varDelta } \to \mathcal{A}(G)$. It is injective, and it induces a homeomorphism from $\hat{\varDelta }$ to a closed subgroup of $\hat{\varGamma }$.
Under the additional condition that the trace field of $\varDelta$ is the same as the invariant trace field this closed subgroup is in fact open by superstrong approximation, cf.\ \cite[Proposition~4.4]{KucharczykActa}.

For a subgroup $\varGamma\in\Sigma (G)$ let $C_{\varphi^{-1}(\varGamma )}$ be the complex curve with $C_{\varphi^{-1}(\varGamma )}(\bbc )=\varphi^{-1}(\varGamma )\backslash\bbh$. Then $\tif$ descends to a morphism of algebraic varieties
$$f_{\varGamma }\colon C_{\varphi^{-1}(\Gamma )}\to \Sh_{\varGamma }^0(G,X).$$
Passing to limits in the category of $\bbc$-schemes, we obtain $\bbc$-schemes of infinite type
$$\hat{C}\defined\varprojlim_{\varGamma }C_{\varphi^{-1}(\varGamma )}\qquad\text{and}\qquad\Sh^0(G,X)\cong\varprojlim_{\varGamma }\Sh_{\varGamma }^0(G,X).$$
By construction, $\hat{\varDelta }$ acts continuously by $\bbc$-scheme isomorphisms on $\hat{C}$, and $\mathcal{A}(G)$ acts similarly on~$\Sh^0(G,X)$ (as we have already seen).

We summarise these considerations:
\begin{Proposition}
A modular embedding as above gives rise to a morphism of schemes of infinite type over~$\bbc$,
$$\hat{f}\colon\hat{C}\to\Sh^0(G,X),$$
which is equivariant for the embedding $\hat{\varphi }\colon\hat{\varDelta }\to\mathcal{A}(G)$.

For a congruence subgroup $\varDelta '\subseteq\varDelta$ with closure $\hat{\varDelta }'$ in $\hat{\varDelta }$ the corresponding curve $C'$ can be retrieved as $C'\cong\hat{\varDelta }'\backslash\hat{C}$, and similarly for $\Sh^0(G,X)$.
\end{Proposition}

\section{Galois conjugation}

\noindent In this section we study how abstract field automorphisms of $\bbc$ act on Shimura varieties, automorphic Higgs bundles and modular embeddings. In the end we draw some consequences for dessins d'enfants.

The action of field automorphisms of $\bbc$ on Shimura varieties and automorphic bundles can be described in elementary terms for Shimura varieties of type PEL, but for more general Shimura varieties such as the ones we are considering a more elaborate approach is needed. This approach constructs, for each connected Shimura datum $(G,X)$, each field automorphism $\tau\in\Aut\bbc$ and each special point $x\in X$, a new Shimura datum $(\taux G,\taux X)$ which is obtained from the original one by twisting the groups involved by certain torsors that can be described by Galois cohomology. We give a review of this theory adapted to our purposes in sections \ref{Subs:STG} to~\ref{Subs:MilneShihTheoryAutom}.

\subsection{The Serre, Taniyama and Galois groups}\label{Subs:STG}

Fundamental to the general theory of Galois conjugation on Shimura varieties is a short exact sequence
\begin{equation}\label{SequenceSTG}
1\to\mathfrak{S}\to\mathfrak{T}\to\mathfrak{G}\to 1
\end{equation}
of pro-algebraic groups (i.e., projective limits of linear algebraic groups) over~$\bbq$, equipped with certain splitting data over $\bbc$ and $\bbq_p$. We now review the construction of these three groups, the sequence (\ref{SequenceSTG}) and various auxiliary structures.

\subsubsection*{Elementary construction of the groups}

The group $\mathfrak{G}$ is simply the \emph{absolute Galois group} $\Gal (\qbar /\bbq )$, interpreted as a constant profinite group scheme over~$\bbq$. That is,
$$\mathfrak{G}=\varprojlim_L\left( \Gal (L/\bbq )\right)_{\bbq },$$
where $L$ runs through all finite Galois extensions of $\bbq$ contained in $\qbar$, and $G_{\bbq }$ denotes the constant group scheme  $\Spec\bbq^G$ for any finite group~$G$.

The group $\mathfrak{S}$ is the \emph{Serre group} (also known as Serre torus); it can also be constructed as a projective limit
\begin{equation}\label{SerreGroupAsLimitOfTori}
\mathfrak{S}=\varprojlim_L\mathfrak{S}^L.
\end{equation}
The facts mentioned below are proved in the opening section of~\cite{MilneShih1982}.

The group $\mathfrak{S}^L$ is the quotient of the torus $\bbt^L=\Res_{L/\bbq }\bbg_m$ by the Zariski closure of any sufficiently small finite index subgroup of $\mathfrak{o}_L^{\times }\subset L^{\times }=\bbt^L(\bbq )$. The character group of $\bbt^L$ is the free $\bbz$-module spanned by $[\varrho ]$, where $\varrho$ runs through all field embeddings $\varrho\colon L\to\qbar$, with Galois action $\sigma\cdot [\varrho ]=[\sigma\circ\varrho ]$. Hence the character group of $\mathfrak{S}^L$ is a Galois-stable subgroup of this group; an easy argument involving Minkowski's Geometry of Numbers shows that $\mathrm{X}^{\ast }(\mathfrak{S}^L)$ consists precisely of the characters $\chi\in \mathrm{X}^{\ast }(\bbt^L)$ with
\begin{equation}\label{SerreCondition}
(\sigma -1)(\iota +1)\chi =(\iota +1)(\sigma -1)\chi =0\text{ for all }\sigma\in\Gal (\qbar /\bbq ).
\end{equation}
This condition can be reformulated as follows:
\begin{equation}\label{varSerreCondition}
\mathrm{X}^{\ast }(\mathfrak{S}^L)=\{\sum_{\varrho\colon L\to\qbar }a_{\varrho }[\varrho ]\mid -a_{\varrho }-a_{\iota\varrho }\text{ independent of }\varrho \} .
\end{equation}
There is a natural homomorphism $h^L\colon\bbs\to\mathfrak{S}^L_{\bbr }$ given on character groups by
\begin{equation*}
\mathrm{X}^{\ast }(h^L)\colon \mathrm{X}^{\ast }(\mathfrak{S}^L)\to \mathrm{X}^{\ast }(\bbs )=\bigoplus_{\varrho '\colon\bbr\to\bbc }\bbz [\varrho '],\qquad \sum_{\varrho }a_{\varrho }[\varrho ]\mapsto a_{\mathrm{id }}[\mathrm{id}]+a_{\iota }[\iota ].
\end{equation*}
Composition with the `weight' and `Hodge filtration' homomorphisms yields homomorphisms $w^L\colon\bbg_m\to\mathfrak{S}^L$ (nota bene: defined over $\bbq$) and $\mu^L\colon\bbg_{m,\bbc }\to\mathfrak{S}^L_{\bbc }$ given on characters by
\begin{equation*}
\mathrm{X}^{\ast }(w^L)\colon \mathrm{X}^{\ast }(\mathfrak{S}^L)\to \bbz ,\qquad \sum_{\varrho }a_{\varrho }[\varrho ]\mapsto -a_{\id }-a_{\iota }=\text{ the constant from }(\ref{varSerreCondition}),
\end{equation*}
and
\begin{equation*}
\mathrm{X}^{\ast }(\mu^L)\colon \mathrm{X}^{\ast }(\mathfrak{S}^L)\to \bbz ,\qquad \sum_{\varrho }a_{\varrho }[\varrho ]\mapsto a_{\mathrm{id}}.
\end{equation*}
\begin{Definition}
Let $T$ be a $\bbq$-torus which splits over $L$, and let $\mu\colon\bbg_{m,\bbc }\to T_{\bbc }$ be a cocharacter. Then $\mu$ satisfies the \emph{Serre condition} if
\begin{equation*}
(\sigma -1)(\iota +1)\mu =(\iota +1)(\sigma -1)\mu =0\qquad\text{for all }\sigma\in\Gal (\qbar /\bbq )
\end{equation*}
(compare (\ref{SerreCondition})).
\end{Definition}
A simple calculation shows that the pair $(\mathfrak{S}^L,\mu^L)$ is universal for tori split over $L$ with a cocharacter satisfying the Serre condition:
\begin{Lemma}\label{UniversalPropertySerreTorusFiniteStep}
Let $T$ be a $\bbq$-torus which splits over $L$, and let $\mu\in \mathrm{X}_{\ast }(T)$ be a cocharacter satisfying the Serre condition. Then there exists a unique morphism of tori $\alpha\colon\mathfrak{S}^L\to T$ (defined over $\bbq$) such that $\mu =\alpha_{\bbc }\circ\mu^L$.
\end{Lemma}

For each finite extension $L\subseteq L'\subset\qbar$ the norm map $\mathrm{N}_{L'/L}$ restricts to a homomorphism $\mathfrak{S}^{L'}\to\mathfrak{S}^L$, and the limit (\ref{SerreGroupAsLimitOfTori}) is to be understood with respect to these homomorphisms. Furthermore we can restrict ourselves to CM-fields $L$ because $\mathfrak{S}^L\cong\mathfrak{S}^{L'}$ if $L$ is the union of all $\mathrm{CM}$- or totally real subfields of~$L'$. Hence $\mathfrak{S}$ is the pro-torus whose character group $\mathrm{X}^{\ast }(\mathfrak{S})$ consists of all functions $n\colon\Gal (\qbar /\bbq )\to\bbz$ which factor through some $\Gal (L/\bbq )$ with $L/\bbq$ Galois and CM and for which $-n(\sigma )-n(\iota\sigma )$ is independent of $\sigma$. Since the homomorphisms $h^L$, $w^L$ and $\mu^L$ are compatible with this limiting process, we obtain homomorphisms
\begin{equation*}
h_{\can }=\lim_Lh^L\colon\bbs\to\mathfrak{S}_{\bbr },\qquad w_{\can }\colon\bbg_m\to\mathfrak{S},\qquad \mu_{\can }\colon\bbg_{m,\bbc }\to\mathfrak{S}_{\bbc }.
\end{equation*}
From Lemma~\ref{UniversalPropertySerreTorusFiniteStep} we deduce that $(\mathfrak{S},\mu_{\can })$ is universal among pairs $(T,\mu )$ where $T$ is a torus over $\bbq$ and $\mu\in \mathrm{X}_{\ast }(T)$ satisfies the Serre condition.

The group $\mathfrak{T}$ is called the \emph{Taniyama group}. It can also be constructed as a limit object. In \cite[section~5]{MR546619} Langlands constructs, for each finite Galois extension $L/\bbq$ contained in $\qbar$, an exact sequence of groups
\begin{equation}\label{STGonlevelL}
1\to\mathfrak{S}^L\to\mathfrak{T}^L\to\mathfrak{G}^L\to 1,
\end{equation}
where $\mathfrak{G}^L$ is the constant group $\bbq$-scheme defined by $\Gal (L^{\mathrm{ab}}/\bbq )$, and then shows that (\ref{STGonlevelL}) admits a limit as $L$ varies.

\subsubsection*{Tannakian theory and motives}

The sequence (\ref{SequenceSTG}) also admits another description in terms of Tannakian categories. This is much less amenable to explicit calculations, but it is conceptually very elegant and explains the appearance of this sequence when field automorphisms are applied to Shimura varieties. For the formalism of Tannakian categories and fibre functors see \cite{DeligneMilne1982}. The theory to be summarised below is beautifully explained in~\cite{MR1044823}.

Let $\mathbf{CM}_{\bbc }$ be the Tannakian subcategory of the category of pure Hodge $\bbq$-structures generated by all $\Hup_1(A^{\an },\bbq )$, where $A$ is an abelian variety with complex multiplication over $\bbc$, and the Tate Hodge structure $\bbq (1)$. Then the Tannakian fundamental group of $\mathbf{CM}_{\bbc }$ for the fibre functor given by the underlying $\bbq$-vector spaces is canonically isomorphic to $\mathfrak{S}$, see \cite{Deligne1982} or \cite[Prop.~4.5]{MR1044823}.

We let $\mathbf{Art}_{\bbq }$ be the category of Artin motives, i.e.\ of finite-dimensional continuous $\bbq$-representations of $\Gal (\qbar /\bbq )$. Then the Tannakian fundamental group of $\mathbf{Art}_{\bbq }$ with respect to the tautological fibre functor is canonically isomorphic to~$\mathfrak{G}$.

Finally we let $\mathbf{CM}_{\bbq }$ be the Tannakian category of motives over $\bbq$ generated by the Tate motive $\bbq (1)=\Hup_2(\bbp^1)$, the motives $\Hup_1(A)$ for all abelian varieties $A$ over $\bbq$ with potential complex multiplication, and the Artin motives (which can be viewed as subobjects of the motives $\Hup_0(V)$, where $V$ runs over all finite $\bbq$-varieties). All these motives can be constructed unconditionally using the theory of absolute Hodge cycles developed in~\cite{DeligneHodgeCycles}. Then we let $\mathfrak{T}$ be the fundamental group of $\mathbf{CM}_{\bbq }$ for the fibre functor given by the Betti realisation.

The inclusion $\mathbf{Art}_{\bbq }\hookrightarrow\mathbf{CM }_{\bbq }$ and the base change $\mathbf{CM}_{\bbq }\to\mathbf{CM}_{\bbc }$ give rise to the homomorphisms in~(\ref{SequenceSTG}).

\subsubsection*{Splittings}

For every finite prime $\ell$ we can identify $\mathfrak{G}(\bbq_{\ell })$ with $\Gal (\qbar /\bbq )$, hence the $\bbq_{\ell }$-valued points of (\ref{SequenceSTG}) form an exact sequence
\begin{equation*}
1\to\mathfrak{S}(\bbq_{\ell })\to\mathfrak{T}(\bbq_{\ell })\to\Gal (\qbar /\bbq );
\end{equation*}
the last map is in fact also surjective, and there is even a canonical splitting
$$\operatorname{sp}_{\ell }\colon\Gal (\qbar /\bbq )\to\mathfrak{T}(\bbq_{\ell }).$$
It can easily be defined using the Tannakian formalism:

Let $\Hup_B$ be the canonical `Betti' fibre functor $\mathbf{CM}_{\bbq }\to\mathbf{Vect}_{\bbq }$. Then $\Hup_{\ell }=\Hup_B\otimes_{\bbq }\bbq_{\ell }$ is endowed with a Galois representation (arising from the Galois action on Tate modules of abelian varieties). Hence each $\tau\in\Gal (\qbar /\bbq )$ gives rise to an automorphism of $\Hup_{\ell }$, hence to an element of $\mathfrak{T}(\bbq_{\ell })$, and this is by definition $\operatorname{sp}_{\ell }(\tau )$. Note that $\operatorname{sp}_{\ell }$ does not come from a homomorphism of pro-algebraic groups $\mathfrak{G}_{\bbq_{\ell }}\to\mathfrak{T}_{\bbq_{\ell }}$, see~\cite[p.~261]{Deligne1982}.

The product $\operatorname{sp}\colon\Gal (\qbar /\bbq )\to\mathfrak{T}(\bba^f)$ of the $\operatorname{sp}_{\ell }$ is likewise a splitting of the sequence of groups of $\bba^f$-valued points obtained from~(\ref{SequenceSTG}). 

\subsection{Torsors and Galois cohomology}\label{Subs:Torsors}

We recall some basic constructions in Galois cohomology, in both explicit and cohomological descriptions. For details and proofs, see the classical source \cite[Chap.~III, \S 1]{MR0180551} or, in greater detail, the recent~\cite{Emsalem2015}.

Recall that for an algebraic group $G$ over a field $k$, a \emph{(right) torsor} or \emph{principal homogeneous space} is a $k$-variety $P$ together with a right action of $G$ on $P$ such that over an algebraic closure $\widebar{k}$ there exists an isomorphism $P_{\widebar{k}}\to G_{\widebar{k}}$ equivariant for $G_{\widebar{k}}$, which operates on itself by right translations. A torsor is \emph{trivial} if such an isomorphism already exists over $k$; this is the case if and only if $P(k)$ is nonempty.

Let $X$ be a $k$-variety on which $G$ acts from the left. Then the \emph{contracted product} or \emph{twist of $X$ by $P$}, denoted by $P\times^GX$, is the scheme-theoretic quotient $(P\times_kX)/G$ (provided that it exists), where $g\in G(\widebar{k})$ acts by $(p,x)\mapsto (pg^{-1},gx)$. Two special cases shall be of interest:

\begin{enumerate}
\item If $f\colon G\to H$ is a homomorphism of algebraic groups, then $P\times^GH$ (where $H$ is acted upon by $G$ by left multiplication via $f$) is in a natural way a right $H$-torsor, denoted by $f_{\ast }P$.
\item If $X$ has some extra structure respected by the $G$-action (such as a group structure, or a $k$-algebra structure on a variety isomorphic to $\bba^n_k$), then $P\times^GX$ has the same structure.
\end{enumerate}
The $G$-torsors up to isomorphism are classified by the non-abelian Galois cohomology set $\Hup^1(k,G)$. If $G$ is the automorphism group of some `variety with additional structures' $X$, then the same cohomology set classifies forms of $X$, i.e.\ varieties with additional structures $X'$ with $X'_{\widebar{k}}\simeq X_{\widebar{k}}$. This correspondence is realised by sending a torsor $P$ to $P\times^GX$, and pushforward of torsors as in (i) corresponds to pushforward of cohomology classes $f_{\ast }\colon\Hup^1(k,G)\to\Hup^1(k,H)$.

Returning to the exact sequence (\ref{SequenceSTG}), the Tannakian formalism gives rise to a simple construction of a $\mathfrak{T}$-torsor, the \emph{period torsor} $\mathfrak{P}=\underline{\mathrm{Isom }}^{\otimes }(\Hup_{\mathrm{B}},\Hup_{\mathrm{dR}})$. Here $\Hup_{\mathrm{B}}$ is the fibre functor on $\mathbf{CM}_{\bbq }$ defined by the Betti realisation, and $\Hup_{\mathrm{dR}}$ is defined by the de Rham realisation. The comparison isomorphism $\Hup_{\mathrm{B}}\otimes_{\bbq }\bbc \cong\Hup_{\mathrm{dR}}\otimes_{\bbq }\bbc$ defines an element $c\in\mathfrak{P}(\bbc )$, and for each $\tau\in\Aut\bbc$ we let $z_{\infty }(\tau )\in\mathfrak{T}(\bbc )$ be the unique element with
$$c\cdot z_{\infty }(\tau ) =\tau (c)\in\mathfrak{P}(\bbc ).$$
The map $z_{\infty }\colon\Aut\bbc\to\mathfrak{T}(\bbc )$ is a \emph{crossed} homomorphism, and it does not factor through $\Gal (\qbar /\bbq )$.

For $\tau\in\Aut\bbc$ the preimage of $\{\tau |\qbar\}\subset\Gal (\qbar /\bbq )=\mathfrak{G}(\bbq )$ under $\mathfrak{T}\to\mathfrak{G}$ is a $\bbq$-subscheme ${^{\tau }\mathfrak{S}}\subset\mathfrak{T}$ which is by construction a right $\mathfrak{S}$-torsor. In fact, the point $z_{\infty }(\tau )$ lies in ${^{\tau }\mathfrak{S}(\bbc )}\subset\mathfrak{T}(\bbc )$ (see \cite[Prop.~1.1]{MR931206}) and therefore provides a distinguished trivialisation of ${^{\tau }\mathfrak{S}}$ over~$\bbc$. For further reference let $s_{\tau }\in\Hup^1(\bbq ,\mathfrak{S})$ be the cohomology class corresponding to~$^{\tau }\mathfrak{S}$.

\subsection{Galois conjugates of Shimura varieties}\label{Subs:MilneShihTheory}

We recall the action of $\tau\in\Aut\bbc$ on the connected Shimura variety $\Sh^0(G,X)$. The description of this action depends on a special point $x\in X$ that is for now considered arbitrary and will later be specified conveniently.

\subsubsection*{Special points}

Let $E\subset B$ be an $F$-subalgebra which is a CM field extension of~$F$. Then let $\tit\subset G$ be the maximal torus with
\begin{equation*}
\tit (\bbq )=\{ e\in E^{\times }\mid\mathrm{N}_{E/F}(e)=1\} \subset B^1=G(\bbq ),
\end{equation*}
and let
$$T=\tit /\mathcal{Z}(G)\cong\bbt^E/\bbt^F$$
be its image in $\gad =G/\mathcal{Z}(G)$; that is,
\begin{equation*}
T(\bbq )=E^{\times }/F^{\times }\subset B^{\times }/F^{\times }=G(\bbq ).
\end{equation*}
By an \emph{incomplete CM-type} on $E$ we mean a subset $\varPhi\subset\Hom_{\bbq }(E,\bbc )$ such that $\varPhi\cap\iota\varPhi =\varnothing$ and
\begin{equation*}
\{ \varphi |_F^{\bbr }\mid \varphi\in\varPhi \} =\mathscr{P}\subseteq\Hom_{\bbq }(F,\bbr ).
\end{equation*}
There is a unique fixed point $x$ of $T(\bbr )$ on~$X$. The tangent space $\operatorname{Tgt}_xX$ is a complex vector space of dimension $\lvert\mathscr{P}\rvert$ with a linear action of $E^{\times }$ via its quotient $T(\bbq )$. The set $\varPhi$ of embeddings $E\to\bbc$ that appear as a character in the $E^{\times }$-representation on $\operatorname{Tgt}_xX$ is an incomplete CM-type on~$E$, and as an $E^{\times }$-module $\operatorname{Tgt}_xX$ is isomorphic to
$$\bigoplus_{\varphi\in\varPhi }\bbc_{\varphi }.$$

Since $T$ is constructed as a quotient of $\bbt^E=\Res_{E/\bbq }\bbg_m$, its character group becomes identified with a $\Gal (\qbar /\bbq )$-submodule of
\begin{equation*}
\mathrm{X}^{\ast }(\bbt^E)\cong\bigoplus_{\sigma\colon E\to\bbc }\bbz\cdot [\sigma ],
\end{equation*}
namely with
\begin{equation*}
\mathrm{X}^{\ast }(T)\cong \{\sum_{\sigma\colon E\to\bbc }a_{\sigma }[\sigma ]\mid a_{\sigma }\in\bbz ,\, a_{\sigma }+a_{\iota\sigma }=0\text{ for all }\sigma\} .
\end{equation*}
Let $\tilde{h}_x\colon \bbs\to (\bbt^E)_{\bbr }$ be the homomorphism of real algebraic tori given on character groups by
\begin{equation*}
\bigoplus_{\sigma\colon E\to\bbc }\bbz\cdot [\sigma ]\to \mathrm{X}^{\ast }(\bbs )=\bbz\cdot [\mathrm{id}]\oplus\bbz\cdot [\iota ],\qquad [\sigma ]\mapsto\begin{cases}
[\mathrm{id}]-[\iota ]&\text{if }\sigma\in\varPhi ,\\
[\iota ]-[\mathrm{id}]&\text{if }\sigma\in\iota\varPhi ,\\
0&\text{if }\sigma\notin\varPhi\cup\iota\varPhi ,
\end{cases}
\end{equation*}
and let $h_x\colon\bbs\to T_{\bbr }$ be its composition with $\bbt^E\to T$. An easy calculation shows that the composition $h_x\colon\bbs\to T_{\bbr }\subset G_{\bbr }$ is precisely the homomorphism corresponding to the point $x\in X$, when $X$ is interpreted as a $\gad (\bbr )^+$-conjugacy class of homomorphisms $\bbs\to \gad_{\bbr }$.

The corresponding cocharacter $\tilde{\mu }_x\colon\bbg_{m,\bbc }\to (\bbt^E)_{\bbc }$ which defines the Hodge filtration is given by
\begin{equation}\label{LiftOfMuxOnChar}
\mathrm{X}^{\ast }(\bbt^E)\to \mathrm{X}^{\ast }(\bbg_m)=\bbz ,\qquad [\sigma ]\mapsto\begin{cases}
1&\text{if }\sigma\in\varPhi ,\\
-1&\text{if }\sigma\in\iota\varPhi ,\\
0&\text{if }\sigma\notin\varPhi\cup\iota\varPhi ;
\end{cases}
\end{equation}
this uniquely determines $\mu_x\colon\bbg_{m,\bbc }\to T_{\bbc }$. The cocharacter $\mu_x$ (but not necessarily $\tilde{\mu }_x$, see Remark~\ref{Rem:WhatGoesWrong}.(i)) satisfies the Serre condition and hence factors through the Serre torus. That is, there is a unique homomorphism $\alpha_x\colon\mathfrak{S}\to T$ defined over $\bbq$ such that the diagram
\begin{equation*}
\xymatrix{
\bbs \ar[rr]^-{h_x}\ar[rd]_-{h_{\mathrm{can}}} && T_{\bbr }\\
&\mathfrak{S}_{\bbr } \ar[ur]_-{(\alpha_x)_{\bbr }}
}
\end{equation*}
commutes.

\subsubsection*{The twisted group}

The pushforward of the $\mathfrak{S}$-torsor ${^{\tau }\mathfrak{S}}$ along the composition
$$\alpha_x\colon\mathfrak{S}\to T\subset \gad$$
defines a cohomology class $g_{\tau }=\alpha_{x,\ast }(s_{\tau })\in \Hup^1(\bbq ,\gad )$ and, by interpreting $\gad$ as the inner automorphism group of $G$, an inner form $\taux G = {^{\tau }\mathfrak{S}}\times^{\mathfrak{S}}G$ of~$G$. Similarly we obtain the inner form $\taux\gad= {^{\tau }\mathfrak{S}}\times^{\mathfrak{S}}\gad$ of~$\gad$.

But $\gad$ is also the group of $F$-algebra automorphisms of $B$, hence twisting by $\alpha_{x,\ast }({^{\tau }\mathfrak{S}})$ defines an $F$-algebra $\taux B=B\times^{\mathfrak{S}} {^{\tau }\mathfrak{S}}$. It is a form of $B$, hence a quaternion algebra over~$F$; the group $\taux\gad (\bbq )$ can be canonically identified with $\taux B^{\times }/F^{\times }$, and similarly for~$\taux G$. The point $\sp (\tau |_{\qbar } )\in {^\tau\mathfrak{S}}(\bba^f)$ trivialises the torsor $^\tau\mathfrak{S}$ over~$\bba^f$, hence it defines a canonical isomorphism of algebras over $F\otimes_{\bbq }\bba^f\cong\bba_F^f$:
\begin{equation*}
\eta\colon B\otimes_F\bba_F^f\to \taux B\otimes_F\bba_F^f.
\end{equation*}
It can be split into local components
\begin{equation}\label{LocalComparisonIsomorphism}
\eta_v\colon B\otimes_FF_v\to \taux B\otimes_FF_v
\end{equation}
for every finite prime $v$ of~$F$. We will also denote the corresponding isomorphisms
\begin{equation}\label{EtaIsomorphismForSemisimpleGroups}
G_{\bba^f}\to\taux G_{\bba^f}\quad\text{and}\quad \gad_{\bba^f}\to\taux\gad_{\bba^f}
\end{equation}
by $\eta$. There is also a canonical isomorphism of topological groups
\begin{equation}\label{EtaIsomorphismForCongruenceCompletion}
\eta\colon\mathcal{A}(G)\to\mathcal{A}(\taux G)
\end{equation}
compatible with (\ref{EtaIsomorphismForSemisimpleGroups}), constructed in \cite[Lemma~8.2]{MilneShih1982a}.
\begin{Proposition}\label{ExactShapeOfTwistedGroup}
For every finite prime $v$ of $F$, the quaternion algebra $\taux B$ splits over $v$ if and only if $B$ splits over~$v$. The set of infinite primes of $F$ where $\taux B$ splits is equal to $\tau\mathscr{P}$.
\end{Proposition}
\begin{proof}
This is, in a slightly different language, the main result of~\cite{MR0572986}. It is stated in the language we use (but without proof and with a sign error) in \cite[Remark~4.1(b)]{MR1044823}. Note that the statement about the finite primes follows directly from (\ref{LocalComparisonIsomorphism}); for the infinite primes some tedious but elementary calculations are needed. We refer to \cite{MR0572986} for these, but it should be possible to replace them by the explicit cohomological calculations of~\cite{MilneShih1982}.
\end{proof}

\subsubsection*{The twisted Shimura datum}

Let $\taux G$ be the semisimple group constructed before, and let $\taux\gad$ be its adjoint group. Note that the torus $\taux T\subset\taux\gad$ is canonically isomorphic to $T$ (even though the cohomology class $t_{\tau }\in\Hup^1(\bbq ,T)$ may be nontrivial) because the action of $T$ on itself by inner automorphisms is trivial. Hence we may apply $\tau$ to the cocharacter $\mu_x\colon\bbg_{m,\bbc }\to T_{\bbc }$ and interpret the result as a cocharacter $\tau\mu_x\colon\bbg_{m,\bbc }\to\taux T_{\bbc }$. This gives rise to a homomorphism $^{\tau }h\colon\bbs\to\taux \gad_{\bbr }$. We let $\taux X$ be the $\taux\gad(\bbr )^+$-conjugacy class of homomorphisms $\bbs\to\taux\gad_{\bbr }$ containing~$^\tau h$. We denote $^\tau h$, when viewed as a point in $\taux X$, by~$^{\tau }x$. An easy calculation shows that then
$$\taux (G ,X)=(\taux G,\taux X)$$
is again a connected Shimura datum.

Its `classical' description is this: for every $\varrho\in\mathscr{P}$ we are given an isomorphism $\taux B\otimes_{F,\tau\circ\varrho }\bbr\simeq\mathrm{M}_2(\bbr )$, and hence an operation of $\taux\gad (\bbr )$ on $\bbh^{\tau\mathscr{P}}$ by coordinate-wise M\"{o}bius transformations; corresponding to these isomorphisms we also obtain an isomorphism $\taux X\simeq\bbh^{\tau\mathscr{P}}$.

We are now ready to describe exactly how $\tau$ acts on the quaternionic Shimura varieties we consider.

\begin{Theorem}
There exists a unique isomorphism of $\bbc$-schemes
\begin{equation*}
\varphi_{\tau ,x}^0\colon\tau\Sh^0(G,X)\to\Sh^0(\taux G,\taux X)
\end{equation*}
which takes $\tau [x]$ to $[{^{\tau }x}]$ and is equivariant for
$$\eta\colon \mathcal{A}(G)\to\mathcal{A}(\taux G)$$
 as in (\ref{EtaIsomorphismForCongruenceCompletion}).
\end{Theorem}
This is a special case of a result of Milne--Shih~\cite{MR717596}, conjectured by Langlands~\cite{MR546619}. For this case this result had been proved earlier in more classical language in \cite{MR0219537} for Shimura curves and in \cite{MR0572986} for quaternionic Shimura varieties of arbitrary dimension.
The present author prefers the more abstract setup utilised here because it also leads to a useful description of the Galois actions on automorphic bundles, needed later on.

Let $\varGamma\subset G (\bbq )=B^1$ be a congruence subgroup, let $\hat{\varGamma }$ be its congruence completion, i.e.\ its topological closure in $G (\bba^f)$, and let $\taux\varGamma =\taux G (\bbq )\cap\eta (\hat{\varGamma })$. Then $\taux\varGamma$ is a congruence subgroup of $\taux G (\bbq )\simeq (B')^1$, and we obtain an isomorphism of complex algebraic varieties
$$(\varphi_{\tau ,x}^0)_{\varGamma }\colon\tau\Sh_{\varGamma }^0(G,X)\to \Sh_{\taux\varGamma }^0(\taux G,\taux X).$$

\begin{Example}\label{ExampleConjugatesCongruenceSubgroup}
\begin{enumerate}
\item Let $\mathscr{O}\subset B$ be an order, and let $\hat{\mathscr{O}}=\mathscr{O}\otimes_{\bbz }\hat{\bbz }$ be its closure in $B_{\bba^f}=B\otimes_F\bba_F^f$. Then $\hat{\mathscr{O}}$ is a compact open subring of $B_{\bba^f}$, hence $\eta (\hat{\mathscr{O}})$ is a compact open subring of $\taux B_{\bba^f}$. We set
$$\taux\mathscr{O}\defined\eta (\hat{\mathscr{O}})\cap\taux B.$$
This is an order in $\taux B$; it is maximal if and only if $\mathscr{O}$ is maximal.

Then $\varGamma =\mathscr{O}^1$ is a congruence subgroup of $G (\bbq )$ and $\taux\varGamma =(\taux\mathscr{O})^1$ is a congruence subgroup of $\taux G (\bbq )$. If $S$ is the complex algebraic variety with $S(\bbc )=\mathscr{O}^1\backslash\bbh^{\mathscr{P}}$ and $\taux S$ is the complex algebraic variety with $\taux S(\bbc )=(\taux\mathscr{O})^1\backslash\bbh^{\tau\mathscr{P}}$, then the above-mentioned theorem states the existence of a canonical isomorphism $\tau S\cong\taux S$.
\item Continue with the notation from the previous example, and let $\mathfrak{a}$ be an integral ideal of~$F$. Using the fact that $\eta$ is an isomorphism of $\bba_F^f$-algebras we find the formula
$$\taux (\varGamma (\mathfrak{a}))=(\taux\varGamma )(\mathfrak{a})$$
for principal congruence subgroups. Hence if we let $S(\mathfrak{a})$ and $\taux S(\mathfrak{a})$ be the varieties with $S(\mathfrak{a})(\bbc )=\mathscr{O}^1(\mathfrak{a})\backslash X$ and $\taux S(\mathfrak{a})(\bbc )=(\taux\mathscr{O})^1(\mathfrak{a})\backslash\taux X$, then $\tau S(\mathfrak{a})\cong\taux S(\mathfrak{a})$.
\end{enumerate}
\end{Example}

\begin{Remark}\label{Rem:WhatGoesWrong}
\begin{enumerate}
\item The point $x$ is always a special point, but only a CM point if $B$ is totally indefinite. This is not a defect of our construction but necessary: CM points exist only if the weight homomorphism $\bbg_{m,\bbr }\to\bbs\to\gad_{\bbr }$ is defined over~$\bbq$ (compare the discussion in \cite[Section~12]{MR2192012}). In our case the field of definition of the weight homomorphism (which is independent of $x$) is the fixed field in $\qbar$ of
$$\{ \sigma\in\Gal (\qbar /\bbq )\mid \sigma\circ\mathscr{P}=\mathscr{P}\} .$$
Clearly this is equal to $\bbq$ if and only if $\mathscr{P}=\mathscr{S}_{\infty }$, the set of all archimedean places of~$F$.

If this is the case, the map $\varrho_x\colon\mathfrak{S}\to T$ can be lifted to $\bbt^E$ and the above construction simplifies considerably. The cohomology class $\varrho_{x,\ast }(s_{\tau })$ is then the image of a class in $\Hup^1(\bbq ,\bbt^E)\cong \Hup^1(E,\bbg_m)=0$ and therefore the connected Shimura datum $^{\tau ,x}(G,X)$ is isomorphic to $(G,X)$.
\item The choice of $x$ is inessential: changing it amounts to applying certain canonical isomorphisms everywhere, cf.~\cite[Proposition~1.3 and Lemma~5.1]{MR931206}.
\end{enumerate}
\end{Remark}

\subsubsection*{Galois conjugates of automorphic bundles}\label{Subs:MilneShihTheoryAutom}

Let $\taux\check{X}$ be the compact dual symmetric space of $\taux X$. This comes with a natural action of $\taux G_{\bbc }$. Since $\taux X$ is derived from a quaternion algebra $\taux B$ over $F$, we can construct a $\taux G_{\bbc }$-equivariant line bundle $\taux \mathfrak{L}_{\varrho }$ on $\taux\check{X}$, for every $\varrho\colon F\to\bbr$ at which $\taux B$ is unramified; that is, for every $\varrho\in\tau\mathscr{P}$.

Now the trivialisation $z_{\infty }(\tau )\in {^{\tau }\mathfrak{S}(\bbc )}$ defines an isomorphism of complex algebraic groups
\begin{equation}\label{IsomorphismOfGroupsOverC}
G_{\bbc }\to\taux G_{\bbc },\qquad g\mapsto [z_{\infty }(\tau ),g].
\end{equation}
\begin{Proposition}
There is a commutative diagram
\begin{equation*}
\xymatrix{
\tau \mathfrak{L}_{\varrho } \ar[r]^-{\cong }\ar[d] & \taux \mathfrak{L}_{\tau\circ\varrho } \ar[d] \\
\tau \check{X} \ar[r]_-{\cong } &\taux\check{X}
}
\end{equation*}
where the vertical maps are the natural structure maps of line bundles, the horizontal isomorphisms are equivariant for (\ref{IsomorphismOfGroupsOverC}) and the lower horizontal map sends $\tau [x]$ to~$[ {^{\tau }x}]$.
\end{Proposition}
\begin{proof}
This is a simple combination of \cite[Proposition~2.7 and Theorem~3.10]{MR931206}, but in our special case it can also be checked by a straightforward calculation.
\end{proof}
The $\taux G_{\bbc }$-line bundle $\taux\mathfrak{L}_{\varrho }$ on $\taux\check{X}$ gives rise to an automorphic line bundle $\taux\mathcal{L}_{\varrho }$ on $\Sh^0(\taux G,\taux X)$, which comes with a continuous $\taux G(\bba^f)$-action, just as for $\mathfrak{L}_{\varrho }$ on~$\check{X}$.
\begin{Theorem}\label{ConjugatesOfAutomorphicLineBundles}
The exists an isomorphism of line bundles $\tau \mathcal{L}_{\varrho }\to\taux \mathcal{L}_{\tau\varrho }$ covering the isomorphism $\varphi_{\tau ,x}^0\colon\tau\Sh^0(G,X)\to\Sh^0(\taux G,\taux X)$ and equivariant for $\eta\colon G(\bba^f)\to\taux G(\bba^f)$.
\end{Theorem}
\begin{proof}
This is a special case of Milne's result \cite[Theorem~5.2]{MR931206}.
\end{proof}

\subsection{Galois conjugates of modular embeddings}

In this subsection we prove Theorem~B, first in the simple version stated in the introduction (Theorem~\ref{ThmConjModularEmbSimple} below) and then in a more elaborate version (Theorem~\ref{ThmConjModularEmbAdelic} below) taking into account the adelic structures.
\begin{Theorem}\label{ThmConjModularEmbSimple}
Let $C$ be a smooth complex curve, let $S=\Sh^0_{\varGamma }(G,X)$ be a quaternionic Shimura variety, with $\varGamma$ torsion-free. Let $f\colon C\to S$ be covered by a modular embedding with respect to $\varrho\colon F\to\bbr$. Let $\tau\in\Aut\bbc$, and identify $\tau S$ with a quaternionic Shimura variety via the isomorphism $(\varphi_{\tau ,x}^0)_{\varGamma }$.

Then $\tau f\colon\tau C\to\tau S$ is covered by a modular embedding with respect to $\tau\circ\varrho$.
\end{Theorem}
\begin{proof}
For projective $C$ this follows easily by combining Theorems \ref{CriteriaModularEmbedding} and~\ref{ConjugatesOfAutomorphicLineBundles}.

For affine $C$ we use Theorem~\ref{GaussManinCriterion} instead. Note that all constructions involved are purely algebraic and stable under field automorphisms: relative de Rham cohomology and the Gauss--Manin connection are, and the Deligne extension of relative de Rham cohomology can be characterised as the unique extension to an algebraic vector bundle with algebraic logarithmic connection having nilpotent residue, cf.\ Remark~\ref{AlgebraicStructuresOnHodgeSubbundles}.
\end{proof}

\subsubsection*{Field automorphisms and uniformisation for curves}

Let $\varDelta\subset\PGL_2(\bbr )^+$ be a cocompact lattice, and let $\tau\in\Aut\bbc$. Then there exists a lattice ${^{\tau }}\!\varDelta\subset\PGL_2(\bbr )^+$, unique up to $\PGL_2(\bbr )^+$-conjugacy, such that applying $\tau$ to the complex algebraic curve underlying $\varDelta\backslash\bbh$ yields a curve isomorphic to ${^{\tau }}\!\varDelta\backslash\bbh$, where the isomorphism is in addition required to preserve the orders of the elliptic points. This is of course a simple construction, but for our purposes there are two disadvantages to this elementary formulation: the group ${^{\tau }}\!\varDelta$ is only well-defined up to conjugacy, and it does not harmonise with the more elaborate constructions needed to describe Galois conjugates of Shimura varieties.

Therefore we fix a point $y\in\bbh$. We shall construct an hermitian symmetric domain $\tauy\bbh$, a real algebraic group $\tauy\PGL_2$ with an action of $\tauy \PGL_2(\bbr )^+$ on $\tauy\bbh$ by biholomorphisms, and a lattice $\tauy\!\varDelta\subset\tauy\PGL_2(\bbr )^+$.

There is a well-defined smooth curve $C$ over $\bbc$ with $C(\bbc )\cong\varDelta\backslash\bbh$, and there is a map $e\colon C(\bbc )\to\bbn$ sending the image of a point $x\in\bbh$ to the order of the stabiliser $\varDelta_x$. Let $c\in C(\bbc )$ be the image of $y\in\bbh$. Then $(\tau C,\tau c)$ is again a smooth complex curve equipped with a geometric point, and we define $(\tauy\bbh,{^{\tau }}y)\to ((\tau C)^{\an },\tau c)$ to be a ramified covering space which has ramification order $e(\tau^{-1}(p))$ at $p\in \tau C(\bbc )$ and which is universal with these properties. Then $\tauy\bbh$ inherits a natural complex structure turning it into a bounded symmetric domain, and the Bergman metric is a canonical hermitian metric on $\tauy\bbh$ turning it into a hermitian symmetric domain. It is isomorphic, \emph{but not canonically}, to the upper half plane. The group $B$ of biholomorphisms of $\tauy\bbh$ is a real Lie group, and we let $\mathfrak{b}$ be its Lie algebra. Then there exists a unique connected real algebraic subgroup $\tauy\PGL_2\subset\GL (\mathfrak{b})$ such that $\tauy\PGL_2(\bbr )^+=\operatorname{Ad}B\cong B$ in $\GL (\mathfrak{b})$. By construction, the deck transformation group $\tauy\!\varDelta$ of $\tauy\bbh\to (\tau C)^{\an }$ is a lattice in $\tauy\PGL_2(\bbr )^+=B$, and there is a natural isomorphism $\tau (\varDelta\backslash\bbh)\cong\tauy\!\varDelta\backslash\tauy\bbh$ preserving the orders of elliptic points.
\begin{Proposition}\label{ShapePreserved}
The groups $\varDelta$ and $\tauy\!\varDelta$ are isomorphic as abstract groups. The isomorphism can be chosen in such a way as to send elliptic, parabolic and hyperbolic elements to elements of the same type.
\end{Proposition}
\begin{proof}
The isomorphism type of $\varDelta$ is uniquely determined by the orders of elliptic points, the number of cusps and the genus of the quotient. These are the same as for~$\tauy\!\varDelta$. Elliptic, parabolic and hyperbolic elements can be characterised in purely topological terms.
\end{proof}
There seems to be no canonical way to construct such an isomorphism $\varDelta \to\tauy\!\varDelta$ for all lattices $\varDelta$. Note, however, that the profinite completions of $\varDelta$ and $\tauy\!\varDelta$ are canonically isomorphic since they can be interpreted as \'etale fundamental groups. In Theorem~\ref{ThmConjModularEmbAdelic} below we will see that if $\varDelta$ admits a modular embedding then so does $\tauy\!\varDelta$, and there is a canonical isomorphism between their congruence completions.

\subsubsection*{Modular embeddings, once more}

Suppose there is given a modular embedding consting of the following data: a lattice $\varDelta\subset\PGL_2(\bbr )^+$, a quaternionic connected Shimura datum $(G,X)$, an arithmetic subgroup $\varGamma\subset\gad(\bbq )^+$ containing a group in $\Sigma (G)$, a homomorphism $\varphi\colon\varDelta\to\varGamma$ and a $\varphi$-equivariant holomorphic map $\tif\colon\bbh\to X$.

As before, we obtain a regular map $f\colon C\to S$ between normal varieties with $C^{\an }=\varDelta\backslash\bbh$ and $S=\Sh^0_{\varGamma }(G,X)$, and congruence completions $\hat{\varphi }\colon\hat{\varDelta }\to\mathcal{A}(G)$ and $\hat{f}\colon\hat{C}\to\Sh^0(G,X)$.

To formulate the next theorem, we need to assume one condition:
\begin{Condition}\label{ConditionExistenceSpecialPoint}
There exists a point $y\in\bbh$ such that $x=\tif (y)\in X$ is a special point for~$\varGamma$.
\end{Condition}
This is satisfied, for example, if $y$ is an elliptic fixed point for some group commensurable to~$\varDelta$. In particular, for $\varDelta$ commensurable to a triangle group we can always find such a~$y$. On the other hand, the Andr\'{e}--Oort conjecture for $S$ predicts that as soon as $\varDelta$ is nonarithmetic there are at most finitely many possible $\varDelta$-orbits for~$y$.
\begin{Theorem}\label{ThmConjModularEmbAdelic}
Let $y\in\bbh$ be a point as in Condition~\ref{ConditionExistenceSpecialPoint}.
\begin{enumerate}
\item There exists a modular embedding for $\tauy\!\varDelta$ relative to the field embedding $\tau\circ\varrho$, consisting of a group homomorphism $\tauy\varphi\colon\tauy\!\varDelta\to\taux\varGamma\subset\taux\gad (\bbq )^+$ and a $\tauy\varphi$-equivariant holomorphic map
$\tauy\tif\colon\tauy\bbh\to\taux X.$
\item There exists a commutative diagram of $\bbc$-schemes
\begin{equation*}
\xymatrix{
\tau\hat{C} \ar[r]^-{\tau\hat{f}} \ar[d]^-{\cong } & \tau\Sh^0(G,X) \ar[d]_-{\cong } \\
\widehat{\tauy C} \ar[r]_-{\widehat{\tauy f}} &\Sh^0(\taux G,\taux X)
}
\end{equation*}
equivariant for the commutative diagram of topological groups
\begin{equation*}
\xymatrix{
\hat{\varDelta} \ar[r]^-{\hat{\varphi }} \ar[d]^-{\cong } & \mathcal{A}(G) \ar[d]_-{\cong }^{\eta } \\
\widehat{\tauy\!\varDelta } \ar[r]_-{\widehat{\tauy\varphi }} &\mathcal{A}(\taux G).
}
\end{equation*}
\end{enumerate}
\end{Theorem}
\begin{proof}
The morphism $\tau f\colon\tau C\to \tau S$ admits a unique holomorphic lift to the simply-connected covering spaces
$$\tauy\tif\colon\tauy\bbh\to\tauy X$$
such that $\tauy\tif ({^{\tau }y})={^{\tau }x}$. It is equivariant for a homomorphism of deck transformation groups
$$\tauy\varphi\colon\tauy\!\varDelta\to\taux\varGamma .$$
By Theorem~\ref{ThmConjModularEmbSimple} applied to suitable finite-index subgroups, $\tauy\tif$ and $\tauy\varphi$ constitute a modular embedding.

It remains to construct the isomorphisms $\tau\hat{C}\to\widehat{\tauy C}$ and $\hat{\varDelta}\to\widehat{\tauy\!\varDelta }$. To do this, note that a covering $C'\to C$ coming from a finite-index subgroup $\varDelta '\subset\varDelta$ is a congruence covering if and only it arises as the pullback along $f\colon C\to S$ of a congruence covering of $S$. Since the latter are stable under $\tau$, congruence coverings of $C$ correspond bijectively to congruence covers of $\tau C=\tauy C$, and we obtain the desired isomorphisms.
\end{proof}

\subsubsection*{Congruence subgroups}

Assume now in addition that (\ref{MELiesInImageOfGQ}) holds. Recall that both triangle groups and fundamental groups of Teichm\"{u}ller curves admit modular embeddings satisfying~(\ref{MELiesInImageOfGQ}). Then $\varphi$ can be lifted to a homomorphism $\tilde{\varphi }\colon\tilde{\varDelta }\to G(\bbq )$ where $\tilde{\varDelta }\subset\SL_2(\bbr )$ is the preimage of $\varDelta$ and where
$$\tr\tilde{\delta }=\varrho (\tr\tilde{\varphi }(\tilde{\delta }))\qquad\text{for all }\tilde{\delta }\in\tilde{\varDelta }.$$
By Proposition~\ref{InclusionOfTraceFields}.(ii) the trace field $K=\bbq (\tr\varDelta )$ is then contained in~$\varrho (F)$. By restricting to a subalgebra of $B$ we will assume in addition that $K=\varrho (F)$; this is an inessential restriction but it leads to a simpler notation. We obtain a quaternion algebra $A=K\langle\tilde{\varDelta }\rangle\subset\mathrm{M}_2(\bbr )$ over $K$ and an order $\mathscr{O}=\mathfrak{o}_K\langle\tilde{\varDelta }\rangle$ of $A$. The map $\tilde{\varphi }$ extends to an isomorphism $A\to B$ of $\bbq$-algebras which is semilinear for $(\varrho |^K)^{-1}\colon K\to F$, and it sends $\mathscr{O}$ to an order $\mathscr{O}_B$ of~$B$.

We may then set $\tilde{\varGamma }=\mathscr{O}_B^1$ and let $\varGamma$ be its image in $\gad (\bbq )^+$. Then  $\varphi$ may be considered as a homomorphism $\varphi\colon\varDelta\to\varGamma$. For every ideal $\mathfrak{n}$ of $F$ we set
$$\tilde{\varGamma }(\mathfrak{n})=\{\tilde{\gamma }\in\tilde{\varGamma }\mid\tilde{\gamma }-1\in \mathfrak{n}\mathscr{O}_B\} ,$$
and for an ideal $\mathfrak{a}$ of $K$ we set
$$\tilde{\varDelta }(\mathfrak{a})=\{\tilde{\delta }\in\tilde{\varGamma }\mid \tilde{\delta }-1\in \mathfrak{a}\mathscr{O}\} =\tilde{\varphi }^{-1}(\tilde{\varGamma }(\varrho^{-1}\mathfrak{a})).$$
We let $\varGamma (\mathfrak{n})$ and $\varDelta (\mathfrak{a})$ be the images of these groups $\gad (\bbq )^+$ and $\PGL_2(\bbr )^+$, respectively.

\begin{Corollary}\label{Cor:ConjugationOfCSG}
Let $\varDelta\subset\PGL_2(\bbr )^+$ be a lattice admitting a modular embedding with~(\ref{MELiesInImageOfGQ}) and with trace field~$K$. Let $y\in\bbh$ and $\tau\in\Aut\bbc$, and let $\mathfrak{a}$ be an ideal in $K$. Then the trace field of $\tauy\!\varDelta$ is $\tau (K)$, and $\tauy (\varDelta (\mathfrak{a}))=(\tauy\!\varDelta )(\tau\mathfrak{a})$.
\end{Corollary}
\begin{proof}
An easy argument shows that $\hat{\mathscr{O}}=\hat{\mathfrak{o}}_F\langle\hat{\varphi}(\hat{\varDelta })\rangle$ and therefore $\taux\hat{\mathscr{O}}=\hat{\mathfrak{o}}_F\langle\widehat{\tauy\varphi }(\widehat{\tauy\!\varDelta })\rangle$. This implies $\mathfrak{o}_F\langle\tauy\varphi (\tauy\!\varDelta )\rangle$ is dense in $\taux\hat{\mathscr{O}}$; since an order in a quaternion algebra is uniquely determined by its adelic closure we obtain
$$\taux\mathscr{O}=\mathfrak{o}_K\langle\tauy\varphi (\tauy\!\varDelta )\rangle .$$

Now $\widehat{\varDelta (\mathfrak{a})}=\hat{\varphi }^{-1}(\widehat{\varGamma (\mathfrak{a})})$ and
\begin{equation*}
\widehat{\tauy (\varDelta (\mathfrak{a}))}=\widehat{\tauy\varphi }^{-1}(\widehat{\taux (\varGamma (\mathfrak{a}))}) \overset{(\ast )}{=}\widehat{\tauy\varphi }^{-1}((\widehat{\taux\varGamma })(\mathfrak{a}))=(\widehat{\tauy\!\varDelta })(\tau\mathfrak{a}),
\end{equation*}
where $(\ast )$ is justified by Example~\ref{ExampleConjugatesCongruenceSubgroup}. Since $\varDelta$ and $\tauy\!\varDelta$ are dense in their respective congruence completions, this implies that $\tauy (\varDelta (\mathfrak{a}))=(\tauy\!\varDelta )(\mathfrak{a})$.
\end{proof}

\subsection{Application to triangle groups and dessins d'enfants}\label{Sect:ThmC}

Let $\varDelta_{p,q,r}$ be a Fuchsian triangle group, let $\mathfrak{a}$ be an ideal in its trace field and let $X_{p,q,r}(\mathfrak{a})$ be the smooth projective complex curve with associated Riemann surface $\varDelta_{p,q,r} (\mathfrak{a})\backslash\bbh$. There is a unique isomorphism $X_{p,q,r}(1)\cong\bbp^1$ sending the elliptic points of order $p,q,r$ to the points $0,1,\infty$, respectively. Hence the quotient map $X_{p,q,r}(\mathfrak{a})\to X_{p,q,r}(1)\cong\bbp^1$ is a normal \emph{Bely\u{\i} map}: it is unramified outside $0,1,\infty$.

Bely\u{\i} maps are famously in one-to-one correspondence with the combinatorial-topologi\-cal objects known as \emph{dessins d'enfants}; the Galois action on dessins d'enfants is much studied for its importance in anabelian geometry. Let $\mathcal{D}_{p,q,r}(\mathfrak{a})$ be the dessin associated with the aforementioned Bely\u{\i} map.
\begin{Theorem}\label{Theo:ActionOnCongruenceDessins}
The action of $\Aut\bbc$ on the curves $X_{p,q,r}(\mathfrak{a})$ and the dessins $\mathcal{D}_{p,q,r}(\mathfrak{a})$ factors through $\Gal (K_{p,q,r}/\bbq )$; it is given by
\begin{equation*}
\tau X_{p,q,r}(\mathfrak{a})\cong X_{p,q,r}(\tau\mathfrak{a})\qquad\text{and}\qquad\tau\mathcal{D}_{p,q,r}(\mathfrak{a})\cong\mathcal{D}_{p,q,r}(\tau\mathfrak{a}).
\end{equation*}
\end{Theorem}
\begin{proof}
The triangle group $\varDelta_{p,q,r}$ is a lattice in $\PGL_2(\bbr )^+$ with presentation
$$\varDelta_{p,q,r}=\langle x,y,z\mid x^p=y^q=z^r=xyz=1\rangle .$$
Moreover, it is the only lattice in $\PGL_2(\bbr )^+$ with that presentation, up to conjugacy. Therefore if $\tau\in\Aut\bbc$ and $y\in\bbh$, and if we choose a biholomorphism $\tauy\bbh\cong\bbh$ and hence an isomorphism $\tauy\PGL_{2,\bbr }\cong\PGL_{2,\bbr }$, then by Proposition~\ref{ShapePreserved} $\tauy\!\varDelta_{p,q,r }$ must be conjugate in $\PGL_2(\bbr )^+$ to $\varDelta_{p,q,r}$, and we may assume the identifications are chosen in such a way that
\begin{equation*}
\tauy\!\varDelta_{p,q,r}=\varDelta_{p,q,r}.
\end{equation*}
By Corollary~\ref{Cor:ConjugationOfCSG} we obtain
\begin{equation*}
\tauy (\varDelta_{p,q,r}(\mathfrak{a}))=\varDelta_{p,q,r}(\tau\mathfrak{a}).\qedhere
\end{equation*}
\end{proof}
Special cases of this theorem were known before. In \cite[Satz~2.5.1]{Feierabend2008} it was proved for signatures of the form $(2,q,r)$ satisfying some additional regularity property and for prime ideals~$\mathfrak{a}$. Furthermore it was proved for all cocompact arithmetic triangle groups except $\varDelta_{3,4,6}$ in \cite[Theorem~6]{GirondoTorresTeigellWolfart2014}. Note that this is only contained in the preprint version \cite{GirondoTorresTeigellWolfart2014}, not in the published version~\cite{MR3289639}. These two works used a different method, analysing the operation of $\varDelta /\!\varDelta (\mathfrak{a})$ on the space of holomorphic one-forms on $X(\mathfrak{a})$.

The groups $\varDelta_{p,q,r}(\mathfrak{a})$ and the curves $X_{p,q,r}(\mathfrak{a})$ are studied in detail in~\cite{ClarkVoight2015}. In \cite{KucharczykVoight2015} the arithmetic aspects of modular embeddings for triangle groups and their congruence subgroups will be studied further.

\providecommand{\bysame}{\leavevmode\hbox to3em{\hrulefill}\thinspace}
\providecommand{\MR}{\relax\ifhmode\unskip\space\fi MR }
\providecommand{\MRhref}[2]{%
  \href{http://www.ams.org/mathscinet-getitem?mr=#1}{#2}
}
\providecommand{\href}[2]{#2}

\end{document}